\documentclass[onefignum,onetabnum]{siamonline190516}

\newcommand{\mat}[1]{\mathbf{#1}}
\newcommand{\vc}[1]{\mathbf{#1}}
\newcommand{\truMat}{\mat{A}}
\newcommand{\sampMat}{\hat{\truMat}}
\newcommand{\rhs}{\vc{b}}
\newcommand{\truSol}{\vc{x}}
\newcommand{\sampSol}{\hat{\truSol}}
\newcommand{\dimension}{n}
\newcommand{\R}{\mathbb{R}}

\newcommand{\N}{\mathbb{N}}
\newcommand{\Rdim}{\R^{\dimension}}
\newcommand{\Rdimdim}{\R^{\dimension \times \dimension}}
\newcommand{\impSol}{\tilde{\truSol}}
\newcommand{\impSolAt}[1]{\impSol_{#1}}
\newcommand{\pdcone}{S(\Rdim)_+}

\newcommand{\errBayes}[2]{\mathcal{E}_{#2}(#1)}

\newcommand{\E}{\mathbb{E}}
\newcommand{\Pb}{\mathbb{P}}
\newcommand{\normMat}{\mat{B}}
\newcommand{\id}{\mat{I}}
\newcommand{\errMat}{\hat{\mat{Z}}}
\newcommand{\augMat}{\hat{\mat{K}}}
\newcommand{\augFac}{\beta}
\newcommand{\optAugFac}{\augFac^*}

\newcommand{\prior}{P}

\newcommand{\autoCor}{\mat{R}}

\newcommand{\bootstrapAugFac}{\tilde{\augFac}^*}

\newcommand{\augOp}{\tilde{\truMat}}
\newcommand{\augOpAt}[1]{\augOp_{#1}}
\newcommand{\params}{\omega}
\newcommand{\paramSpace}{\Omega}
\newcommand{\paramDist}[1]{\mathbb{P}_{#1}}
\newcommand{\paramToMat}{\mathcal{M}}
\newcommand{\truParam}{\params^*}
\newcommand{\matDistAt}[1]{D_{#1}}
\newcommand{\matDist}{\matDistAt{\truParam}}
\newcommand{\sampParam}{\hat{\params}}

\newcommand{\gl}{GL(\mathbb{R}^n)}
\newcommand{\transErrMat}{\hat{\mat{Y}}}
\newcommand{\errAtom}{\mat{Z}}

\DeclareMathOperator*{\tr}{tr}

\usepackage{lipsum}
\usepackage{amsfonts}
\usepackage{graphicx}
\usepackage{epstopdf}
\usepackage{algorithmic}
\ifpdf
  \DeclareGraphicsExtensions{.eps,.pdf,.png,.jpg}
\else
  \DeclareGraphicsExtensions{.eps}
\fi

\usepackage{enumitem}
\setlist[enumerate]{leftmargin=.5in}
\setlist[itemize]{leftmargin=.5in}

\newsiamremark{remark}{Remark}
\newsiamremark{hypothesis}{Hypothesis}
\crefname{hypothesis}{Hypothesis}{Hypotheses}
\newsiamthm{claim}{Claim}

\headers{Operator Shifting for General Noisy Matrix Systems}{P.A. Etter, L. Ying}

\title{Operator Shifting for General Noisy Matrix Systems}

\author{Philip A. Etter \thanks{Institute for Computational and Mathematical Engineering, Stanford University} (\email{paetter@stanford.edu}) \and Lexing Ying \thanks{Department of Mathematics, Stanford University} (\email{lexing@stanford.edu})}

\usepackage{amsopn}

\begin{document}

\maketitle

\begin{abstract}
In the computational sciences, one must often estimate model parameters from data subject to noise and uncertainty, leading to inaccurate results. In order to improve the accuracy of models with noisy parameters, we consider the problem of reducing error in a linear system with the operator corrupted by noise. Our contribution in this paper is to extend the elliptic \emph{operator shifting} framework from Etter, Ying '20 to the general nonsymmetric matrix case. Roughly, the operator shifting technique is a matrix analogue of the James-Stein estimator. The key insight is that a shift of the matrix inverse estimate in an appropriately chosen direction will reduce average error. In our extension, we interrogate a number of questions --- namely, whether or not shifting towards the origin for general matrix inverses always reduces error as it does in the elliptic case. We show that this is usually the case, but that there are three key features of the general nonsingular matrices that allow for adversarial examples not possible in the symmetric case. We prove that when these adversarial possibilities are eliminated by the assumption of noise symmetry and the use of the residual norm as the error metric, the optimal shift is always towards the origin, mirroring results from Etter, Ying '20. We also investigate behavior in the small noise regime and other scenarios. We conclude by presenting numerical experiments (with accompanying source code) inspired by Reinforcement Learning to demonstrate that operator shifting can yield substantial reductions in error.
\end{abstract}

\begin{keywords}
Operator Shifting, Random Matrices, Monte Carlo, Polynomial Expansion, Asymmetric Matrices.
\end{keywords}

\begin{AMS}
65F99, 62A99, 60B20
\end{AMS}

\section{Introduction}

Numerical linear algebra is a crucial foundation for research across a massive breadth of technical domains. It forms the computational bedrock of everything from data science to computational physics. Even non-linear problems are usually solved via linear approximation. One typically writes such systems via matrix notation,
\begin{equation} \label{eq:base_system}
    \truMat \truSol = \rhs \,,
\end{equation}
where $\truMat \in \Rdimdim$ and $\truSol, \rhs \in \Rdim$ for $\dimension \in \N$.

However, linear systems are often imperfect. Scientific problems can be subject to noise in the underlying data or model parameters, sampling error, or even epistemic uncertainty --- each potentially giving rise to error in predictions or inferences (see, for example, \cite{soize2005comprehensive, palmer2005representing}). So in reality, one is more often confronted by a system
\begin{equation} \label{eq:noisy_system}
    \sampMat \sampSol = \rhs \,,
\end{equation}
where one constructs $\sampMat$ from data to approximate the true $\truMat$. Hence, $\sampSol = \sampMat^{-1} \rhs \in \Rdim$ is the solution one actually obtains when we solving the observed system naively. If the uncertainty or noise is severe enough, the discrepancy between $\sampSol$ and $\truSol$ may be a real practical concern.

These situations are fairly common in the computational sciences. As an example, $\sampMat$ might be a Laplacian for a Markov Chain that is not known outright, but must be sampled via trajectories through the state space. Or perhaps, $\sampMat$ might be the scattering operator through a background that is estimated from data.

Regardless of the specific application, there are a wide variety of techniques available for obtaining a better estimate of the true $\truSol$. For example, one may suppose a certain prior for $\truSol$, as is common in such techniques as Tikhonov regularization \cite{tikhonov1963solution}. In this paper, however, we take a fundamentally different tact. Instead of trying to applying post processing or Bayesian regularization to $\sampSol$, we will instead examine this problem from the standpoint of building an improved estimator for the matrix $\sampMat^{-1}$. The fundamental question we seek to investigate in this paper and the prequel \cite{etter2020operator} is whether there exist operations that can make $\sampMat^{-1}$ potentially more accurate.

\subsection{Operator Shifting}

As this paper is an extension of Etter, Ying '20 \cite{etter2020operator}, some discussion of the previous results of operator shifting is inevitable. We will attempt to give the high level details in this section. 

The fundamental idea of operator shifting is to \emph{shift} the estimator $\sampMat^{-1}$ by an appropriately chosen function $\augMat(\sampMat)$ of $\sampMat$,
\begin{equation}
    \augOpAt{\augFac}^{-1} = \sampMat^{-1} - \augFac \, \augMat \,.
\end{equation}
In continuity with previous work, we refer to $\augMat$ as the \emph{shift matrix} and the scalar quantity $\augFac \in \R$ as the \emph{shift factor}. After choosing the shift matrix, one optimizes $\augFac$ such that the error
\begin{equation}
    \E \| \augOpAt{\augFac}^{-1} - \truMat^{-1} \|^2 \,,
\end{equation}
is minimized with respect to some matrix norm $\|\cdot \|$. The reader will note that performing this optimization is impossible outright, as it requires knowledge of the quantity $\truMat^{-1}$ we are trying to estimate. This issue is not fatal, however, and can be effectively addressed via a bootstrap procedure that we will discuss later.

With regards to the choice of shift matrix $\augMat$, the simplest choice is simply $\augMat(\sampMat) = \sampMat^{-1}$. For $\augFac \in (0, 1)$, this choice corresponds to shrinking the operator towards zero. Indeed, the original intent behind operator shifting was to produce an analogue of the high-dimensional James-Stein \cite{james1992estimation} for matrices. The reasoning involved is that since the underlying space is high-dimensional, the error $\sampMat - \truMat$ will likely be close to orthogonal to $\truMat$ in the Frobenius inner product. Thus, shrinking $\sampMat$ towards the origin logically brings one closer to $\truMat$ in expectation. 

In contrast to the James-Stein setting, however, we must also contend with the presence of the matrix inversion operator $(\cdot)^{-1}$ --- as our goal is to estimate $\truMat^{-1}$ and not $\truMat$. In this respect, there is an extra wrinkle of complexity that one must deal with. 

Fortunately, in the case of positive-definite symmetric matrices matrix inversion has very nice structure. Most importantly, it is convex with respect to the L\"owner order. This means that the naive $\sampMat^{-1}$ will always dominate $\truMat^{-1}$ if $\sampMat$ is unbaised (this is analogous to Jensen's inequality, see \cref{fig:jensen}). The confluence of these two factors --- high-dimensionality and the convexity of $(\cdot)^{-1}$ --- suggest that shrinking towards the origin is the most natural operation to perform on $\sampMat^{-1}$.

\begin{figure}
    \centering
    \includegraphics[scale=0.25]{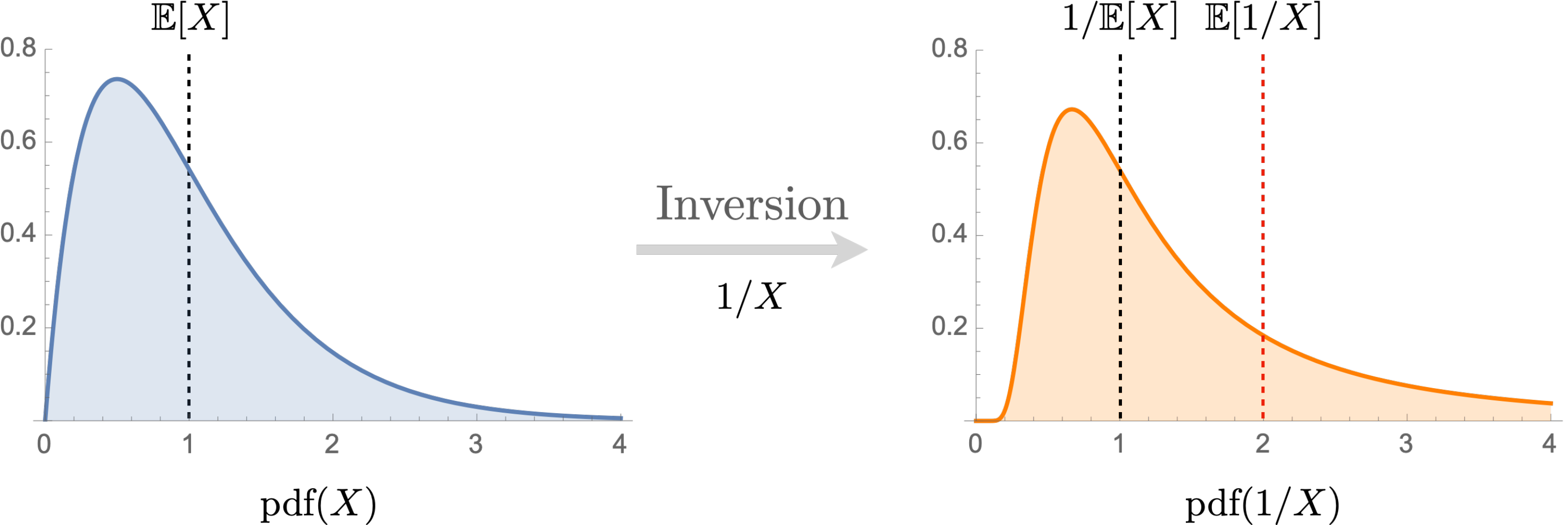}
    \caption{An example of the upward bias induced by inversion. If we take a single sample of the of the scalar random variable $X \sim \Gamma(2, 1/2)$, and invert it, the probability distribution of the $1/X$ has an expectation double that of $1/\E[X]$. Hence, estimating $1/\E[X]$ naively will likely give a significant overestimate. The same principle can apply when $X$ is a random matrix.}
    \label{fig:jensen}
\end{figure}

Indeed, the results of Etter, Ying '20 \cite{etter2020operator} bear this out. In particular, the results of the this paper demonstrate that for positive definite symmetric matrices, shrinking towards the origin always reduces error,

\begin{theorem} \label{thm:etter20}
(Informal, Etter, Ying '20) For all distributions on $\sampMat$ for which $\E[\sampMat] = \truMat$ and $\E[\sampMat^{-2}]$ exists, we have that $\optAugFac \in (0, 1]$ for both the Frobenius norm and residual\footnote{Defined as $\|\mat{X}\|^2 = \| \truMat \mat{X} \|_F^2$} norms.
\end{theorem}

The other primary contributions of Etter, Ying '20 include:

\begin{itemize}
    \item Efficient Monte Carlo estimation of $\augFac$ with monotonic polynomial approximations in the residual norm (to be defined in \cref{sec:gaurantees}).
    \item Lower bounds relating the optimal reduction in error to the variance of the noise in $\sampMat$.
    \item Lower bounds for how far $\optAugFac$ is away from $0$.
\end{itemize}

Of course, the shift $\augMat(\sampMat) = \sampMat^{-1}$ is only one of a huge number of potential shifts. Still, a simple dimensional analysis suggests that $\augMat(\sampMat)$ should always be a homogeneous function of $\sampMat^{-1}$. So, it is natural to consider shifts of the form
\begin{equation} \label{eq:gen_aug}
    \augMat(\sampMat) = \normMat \sampMat^{-1} \autoCor \,,
\end{equation}
for constant matrices $\normMat, \autoCor \in \Rdimdim$. As it turns out, analogues of \cref{thm:etter20} and the bullet points above are provable for this larger class of shifts (as proved in \cite{etter2020operator}). 

\subsection{Novel Contributions and Paper Overview}

We stress that, prior to the work done herein, all of the above applies strictly to \emph{symmetric positive-definite matrices} only. The central question of this paper is \textbf{to what extent the theory of operator shifting for positive-definite matrices can be extended to general nonsingular matrices.}  

In particular, we will interrogate the following questions:
\begin{enumerate}
    \item Is it the case that shrinking the operator $\sampMat^{-1}$ towards zero always produces better results as in the positive definite case? 
    \begin{itemize}
        \item \textbf{If not}, what are the salient structural differences between the positive definitive cone and the general matrix group that allow for counter-examples? 
        \item What do these counter-examples look like?
        \item What are the practical consequences of a potentially negative optimal shift factor?
    \end{itemize}
    \item How large of a class of operator shifts can we extend the theory to? (i.e., does the theory generalize to shifts of the form \cref{eq:gen_aug}?
    \item Does bootstrapped operator shifting on general matrices empirically reduce error?
\end{enumerate}

We partially answer question (1) in \cref{sec:gaurantees}. Our work shows that under certain conditions, it is true that shrinking the operator always produces a reduction in error. We list all of the conditions in \cref{sec:conditions} and give a proof in \cref{sec:proof}. Furthermore, as we have discovered, each one of these conditions has a corresponding illustrative counter-example that both exemplify the critical structural features of the general matrix group and demonstrate the necessity of the aforementioned conditions. We present these examples in \cref{sec:assumptions}. Some of these examples depend on the presence of ``large noise,'' hence, we dedicate \cref{sec:small_noise} to examining what happens when higher order noise terms are negligible. Then, we answer question (2) in \cref{sec:aniso} and provide the main theorem for this paper, \cref{thm:main}.

In \cref{sec:bootstrap}, we introduce the machinery that we will use to approximate the optimal shift factor $\optAugFac$ using a bootstrap optimization procedure. We then use this machinery to give practical algorithms for approximation of $\optAugFac$ using Monte Carlo in \cref{sec:montecarl}. We also provide an adversarial example of how $\optAugFac < 0$ can cause the bootstrapping procedure to fail arbitrarily badly in estimating $\optAugFac$ in \cref{sec:badexample}.

For the numerical experiments section of this paper we draw upon problems from Reinforcement Learning (RL). RL problems frequently require one to approximately solve linear systems (value function estimation for Markov decision processes\footnote{Markov decision processes are Markov processes whose transition probabilities are determined by a controller.}) and linear programs (policy optimization for Markov decision processes). However, the underlying problems can usually only be estimated from data due to both memory and sampling restrictions, making RL the perfect domain in which to apply our technique. Our numerical experiments in \cref{sec:exp} demonstrate that operator shifting can provide substantial error reduction on simple value function estimation problems.

\section{Related Work}

As discussed in the previous section, operator shifting is heavily inspired by the work of James and Stein. Stein's original paper \cite{stein1956inadmissibility} proved the relatively shocking conclusion that in dimensions $\geq 3$, the standard estimator is actually \emph{inadmissible} for the quadratic loss, as shrinking the estimate towards any fixed point by an appropriately chosen amount will always reduce loss in expectation. This idea was later refined by James and Stein in their paper on estimation under quadratic loss \cite{james1992estimation}. Fundamentally, we view our work as taking this idea and applying it to the novel setting of matrices corrupted by noise.

We remark that there are a number of connections with the field of statistical inverse problems. For example, one is often interested in estimating an object from incomplete or noisy measurements. One relevant example is the area of \emph{semi-blind deconvolution}, where one has measurements of a unknown function convolved with a kernel that is known with some uncertainty (as opposed to \emph{blind deconvolution}, where one knows nothing about the kernel). This uncertainty in the underlying operator is a shared feature between our work and this body of literature; however, we should note the both our formalism and the approach are quite different.

The operative approach of related papers in statistical inverse problems tends to be to introduce regularization on both the operator and the recovery target. An example would be the pioneering work of Golub and Van Loan on Total Least Squares (TLS) \cite{golub1980analysis}. Golub and Van Loan optimize over both perturbations to linear features and feature weights themselves, minimizing a residual term together with a regularization on the feature perturbations. Another example includes techniques from \emph{semi-blind deconvolution}, where one introduces a free estimate of the kernel with an appropriate regularization term into the inverse problem optimization \cite{buccini2018semiblind}. Other approaches (i.e., double regularization) involve introducing a free estimate of the operator but constraining the free estimate so that it doesn't differ from the observed operator by too much \cite{bleyer2013double}. In a gross over-simpliciation that we will perform for readability, we will characterize the above approaches as roughly solving a variant of an optimization problem that looks like
\begin{equation}
    \min_{\vc{x}, \mat{E}} \| (\sampMat + \mat{E}) \vc{x} - \vc{b} \|^2 + R_1(\vc{x}) + R_2(\mat{E}) \,,
\end{equation}
where here $\mat{E}$ denotes a correction to the linear features of $\sampMat$, and $R_1$ and $R_2$ denote appropriate regularizers on the recovery target $\vc{x}$ and the matrix correction $\mat{E}$.

But while these works are related, there are stark differences between our respective formalism and approach. The primary difference between our setting and TLS is that we assume we are operating in the regime where $\sampMat$ is non-singular and square, whereas TLS is typically applied in under-determined scenarios. In \emph{semi-blind deconvolution} settings, the choice of regularizer $R_2$ and optimization process both depend heavily on the assumption that the operator $\sampMat$ comes from a kernel convolution. We make no such assumptions about the specific character of $\sampMat$ in our work; though it may certainly be the case that our technique functions better for some problems than others. In addition, other works also do not frame error reduction in terms of producing an estimator for a random matrix in the way that we do here, and hence their analyses are focused more on the optimization methods themselves and less on the underlying probabilistic effects that arise from noisy operators. As a final note, we observe that optimizing over $\mat{E}$ is in many practical scenarios infeasible. On most reinforcement learning problems, for example, even storing the operator in memory would be prohibitively expensive --- however, this fact is something that operator shifting can deal with fairly well, since it only needs to optimize over a single parameter $\augFac$ rather than an entire matrix $\mat{E}$. Moreover, the sheer number of degrees of freedom $\mat{E}$ adds to the optimization in our setting has a danger of contributing to overfitting unless one is careful with regularization.

Other statistical inverse problems literature pertaining to noisy or uncertain operators include situations where the forward operator is too computationally expensive to use in an optimization procedure and is replaced by a learned proxy \cite{lunz2021learned}. This is not directly relevant to our problem at hand, but notable nonetheless. Another situation studied in the literature is when one has a set of noisy input-output pairs of the underlying operator. One can use these input-output pairs to construct regularizer for solving the inverse problem \cite{aspri2020data}. These works are both very different to the approach we take in this paper.

In addition to statistical inverse problems, the field of \emph{model uncertainty} is tangentially relevant. Model uncertainty --- both its quantification and representation --- is an important topic in many branches of computational science, from structural dynamics \cite{soize2005comprehensive} to weather and climate prediction \cite{palmer2005representing}. However, work in model uncertainty is typically tied very closely to a specific domain. In contrast, our work here does not make domain specific assumptions.

Uncertainty quantification (UQ) is another relevant, but ultimately tangential subject. UQ is concerned with quantifying the probability distributions associated with calculations or physical processes. For example, one may be interested in the variance of a set of outputs given a distribution of noise on a set of inputs. Practitioners can quantify this through a variety of means --- including, but not limited to, Monte Carlo techniques \cite{marzouk2016introduction}, stochastic Galerkin projection \cite{xiu2002wiener}, or collocation \cite{xiu2005high}. In our situation, however, we are more interested in the reduction of error rather than quantifying its distribution.

The central problem in this paper is also not too dissimilar to the setting of matrix completion seen in \cite{candes2010matrix, keshavan2009matrix}. In matrix completion, one usually seeks to recover a low-rank ground truth matrix from observations that have been corrupted by additive noise. Regardless, the respective settings of operator shifting and matrix completion are still different. The operator shifting setting operates purely on full-rank matrices, and not those of low-rank.

We should briefly mention that the mathematical branch of random matrix theory (RMT) studies the spectral properties of random matrix ensembles \cite{anderson2010introduction, tao2012topics}. However, RMT results usually apply only when the entries of the random matrices are independent and in the large matrix limit. We find these assumptions to be too stringent for the problem at hand.

In conclusion, we do not believe that the setting we introduce in this paper has been studied in the proposed fashion before. There is little precedent in the literature for the operator shifting method beyond the original paper \cite{etter2020operator}.

\section{Theoretical Guarantees} \label{sec:gaurantees}

In order to provide a theory for nonsymmetric operator shifting that mirrors the theory for symmetric operator shifting, we focus on the inverse operator error
\begin{equation} \label{eq:bayeserr}
    \errBayes{\augOpAt{\augFac}^{-1}}{\normMat, \autoCor} \equiv \E \left[\| \augOpAt{\augFac}^{-1} - \truMat^{-1} \|_{\normMat, \autoCor}^2\right] \,.
\end{equation}

Here, the $\|\cdot\|_{\normMat, \autoCor}$ is a generalized version of a matrix inner product norm for symmetric positive definite $\normMat$ and $\autoCor$. The corresponding matrix inner product, $\langle \cdot, \cdot \rangle_{\normMat, \autoCor}$, we define as:
\begin{equation} \label{eq:normdef}
    \begin{split}
    \langle \mat{X}, \mat{Y} \rangle_{\normMat, \autoCor} &\equiv \tr(\autoCor \mat{X}^T \normMat \mat{Y}) = \tr(\autoCor^{1/2} \mat{X}^T \normMat \mat{Y} \autoCor^{1/2} )\,, \\
    \| \mat{X} \|_{\normMat, \autoCor}^2 &\equiv \langle \mat{X}, \mat{X} \rangle_{\normMat, \autoCor} = \tr((\normMat^{1/2} \mat{X} \autoCor^{1/2})^T (\normMat^{1/2} \mat{X} \autoCor^{1/2})) = \| \normMat^{1/2} \mat{X} \autoCor^{1/2} \|^2_{F} \,,
    \end{split}
\end{equation}
where $\|\cdot\|_F$ denotes the Frobenius norm. Note that when $\normMat$ and $\autoCor$ are the identity, this simply becomes the standard Frobenius inner product. In this way, the above norm is a natural generalization of the Frobenius norm for matrix operators. Just like with the Frobenius norm, one can interpret the $\normMat, \autoCor$ norm via the use expectations. Namely, if $\vc{b} \sim \prior$ is a random vector with second moment matrix $\E[\vc{b} \vc{b}^T] = \autoCor$, then it is the case that:
\begin{equation}
    \| \mat{X} \|^2_{\normMat, \autoCor} = \E_{\vc{b} \sim \prior} \| \mat{X} \vc{b} \|^2_{\normMat} \,,
\end{equation}
where $\| \cdot\|^2_{\normMat}$ is the vector norm induced by the symmetric positive definite matrix $\normMat$, i.e., $\| \vc{x} \|^2_{\normMat} = \vc{x}^T \normMat \vc{x}$. This means that one may interpret \cref{eq:bayeserr} as being the average squared error of the solution of the linear noisy linear system \cref{eq:noisy_system} if the right hand side is sampled from the prior $\prior$. In mathematical notation, we may write:
\begin{equation}
    \errBayes{\augOpAt{\augFac}^{-1}}{\normMat, \autoCor} = \E_{\sampMat} \E_{\rhs \sim \prior} \left[ \| \augOpAt{\augFac}^{-1} \rhs - \truMat^{-1} \rhs \|_\normMat^2 \right] \,.
 \end{equation}

The goal of operator shifting is to approximate the value of $\optAugFac$ that minimizes the above error. Namely,
\begin{equation}
    \optAugFac \equiv \underset{\augFac}{\text{argmin }} \errBayes{\augOpAt{\augFac}^{-1}}{\normMat, \autoCor} \,.
\end{equation}
While exact optimization of this quantity is out of reach for the aforementioned reason that one does not explicitly know $\truMat$, one can develop expectations of how $\optAugFac$ should behave through the use of mathematical theory, and then use bootstrap Monte Carlo to approximate it.

In accordance with our discussion of the previous work on symmetric operator shifting from our introduction, the primary theoretical question we seek out to answer is whether shifting towards the origin (i.e., $\optAugFac > 0$) can be expected to always decrease error as it does in the symmetric positive definite case. 

To begin to answer this question, we perform a simple calculation,
\begin{equation}
    \optAugFac = \frac{\E \langle \augMat, \sampMat^{-1} - \truMat^{-1} \rangle_{\normMat, \autoCor}}{\E \| \augMat\|^2_{\normMat, \autoCor}} \,,
\end{equation}
hence the sign of $\optAugFac$ is equivalent to the sign of $\E \langle \augMat, \sampMat^{-1} - \truMat^{-1} \rangle_{\normMat, \autoCor}$, and we would therefore like an shift method to exhibit $\E \langle \augMat, \sampMat^{-1} - \truMat^{-1} \rangle_{\normMat, \autoCor} \geq 0$. We will begin by studying the simplest choice of operator shift,
\begin{equation} \label{eq:basicaug}
    \augMat = \sampMat^{-1} \,.
\end{equation}
The overshooting effect demonstrated in \cref{fig:jensen} gives one reason to believe that this choice of shift is a reasonable one, as it shrinks the inverse operator towards zero. 

\subsection{Conditions} \label{sec:conditions}

Unlike in the symmetric case, to prove a rigorous statement about the sign of $\optAugFac$ for \cref{eq:basicaug} in the nonsymmetric case, one must place a number of additional conditions on the constituent components of the model. We will discuss the necessity of these conditions in more detail in \cref{sec:assumptions} -- in short, each of these conditions has an adversarial example associated with it that causes the theory to fail when the conditions are not assumed. Throughout the following proofs and discussion we will denote the noise in the matrix $\sampMat$ with the symbol $\errMat$,
\begin{eqnarray}
    \errMat \equiv \sampMat - \truMat \,.
\end{eqnarray}

In order to prove nonnegativity of $\optAugFac$, we introduce the following constraints:

\begin{enumerate}
    \item \textbf{Mean-Zero Noise}: We assume that the noise matrix $\errMat$ is mean-zero, i.e., $\E[\sampMat] = \truMat$.
    \item \textbf{Isotropy}: We assume that $\autoCor = \E_{\prior} [\rhs \rhs^T] = \id$. This means that there is no preferred direction in which we care about the accuracy of the estimator $\augOpAt{\augFac}^{-1}$.
    \item \textbf{Noise Symmetry}: We assume that the distribution of the matrix $\sampMat$ is symmetric about its mean, namely that $\errMat$ has the same distribution as $-\errMat$.
    \item \textbf{Residual Norm}: We specifically choose our norm of interest $\normMat$ to be the residual norm $\normMat = \truMat^T \truMat$. The residual norm is often used as an objective in nonsymmetric iterative methods.
\end{enumerate}

We note that for the theory of elliptic operator shifting, items (2) through (4) are not necessary -- and so their apparent necessity in the nonsymmetric case is a curious mathematical phenomenon of the general nonsingular matrices $\gl$. 

\subsection{Proof} \label{sec:proof} With the conditions outlined above, we proceed to prove the positivity of the shift factor $\optAugFac$,

\begin{theorem} \label{thm:thm}
Let $\sampMat$ be a random matrix, invertible almost everywhere, such that $\E[\sampMat^{-2}]$ exists. Under the conditions outlined in \cref{sec:conditions}, the optimal shift factor is always nonnegative.
\end{theorem}

\begin{proof}
We must verify:
\begin{equation} \label{eq:neededcond}
    \E \langle \sampMat^{-1} , \sampMat^{-1} - \truMat^{-1} \rangle_{\truMat^T \truMat, \id} \geq 0 \,.
\end{equation}
Expanding,
\begin{equation} \label{eq:expanding1}
    \E \langle \sampMat^{-1} , \sampMat^{-1} - \truMat^{-1} \rangle_{\truMat^T \truMat, \id} = \tr \E \left[\sampMat^{-T} \truMat^{T} \truMat (\sampMat^{-1} - \truMat^{-1})\right] \,,
\end{equation}
Since $\truMat$ is invertible and $\sampMat$ is invertible almost everywhere, let us define the matrix $\transErrMat$,
\begin{equation} \label{eq:ydef}
    \transErrMat = \sampMat \truMat^{-1} - \id = \errMat \truMat^{-1} \,.
\end{equation}
Note that this definition implies that $\E[\transErrMat] = \mat{0}$ as well as that the distribution of $\transErrMat$ is symmetric, since the distribution of $\errMat$ is symmetric by condition (3). 

One can rearrange to obtain an expression for $\sampMat^{-1}$,
\begin{equation}
    \sampMat^{-1} = \truMat^{-1} (\id + \transErrMat)^{-1} \,.
\end{equation}
And we substitute the above expression into \cref{eq:expanding1},
\begin{equation}
\begin{split}
     \E \langle \sampMat^{-1} , \sampMat^{-1} - \truMat^{-1} \rangle_{\truMat^T \truMat, \id} &= \tr \E \left[(\id + \transErrMat)^{-T} (\id + \transErrMat)^{-1} - (\id + \transErrMat)^{-T}\right] \,. \\
\end{split}
\end{equation}
Since the distribution of $\transErrMat$ is symmetric, it suffices to verify that
\begin{equation} \label{eq:neededtowork}
    \tr \left[ (\id + \mat{Y})^{-T} (\id + \mat{Y})^{-1} - (\id + \mat{Y})^{-T}\right] + \tr \left[ (\id - \mat{Y})^{-T} (\id - \mat{Y})^{-1} - (\id - \mat{Y})^{-T}\right] \geq 0 \,,
\end{equation}
or alternatively,
\begin{equation} \label{eq:verifygoal}
    \tr \left[ (\id + \mat{Y})^{-T} (\id + \mat{Y})^{-1} +  (\id - \mat{Y})^{-T} (\id - \mat{Y})^{-1} \right] \geq \tr \left[(\id + \mat{Y})^{-T} + (\id - \mat{Y})^{-T}\right] \,,
\end{equation}
for all matrices $\mat{Y}$ for which $\id + \mat{Y}$ and $\id - \mat{Y}$ are nonsingular.

In order to verify \cref{eq:verifygoal}, we begin by considering the matrix,
\begin{equation} \label{eq:strictpos}
    ((\id - \mat{Y})^{-T} - (\id + \mat{Y})^{-1})((\id - \mat{Y})^{-1} - (\id + \mat{Y})^{-T}) \succeq \mat{0} \,.
\end{equation}
Since this matrix is positive definite (by virtue of the fact that it has the form $\mat{M}^T \mat{M}$), it follows that the trace of the above matrix is positive,
\begin{equation}
    \tr \left[ ((\id + \mat{Y})^{-T} - (\id - \mat{Y})^{-1})((\id + \mat{Y})^{-1} - (\id - \mat{Y})^{-T})\right] \geq 0 \,.
\end{equation}
The above inequality can be rearranged,
\begin{equation} \label{eq:almostgoal}
     \tr \left[ (\id + \mat{Y})^{-T} (\id + \mat{Y})^{-1} + (\id - \mat{Y})^{-1}(\id - \mat{Y})^{-T}\right] \geq \tr\left[(\id - \mat{Y}^2)^{-T} + (\id - \mat{Y}^2)^{-1}\right] \,.
\end{equation}
Note that, by the cyclic property of the trace, that the left hand sides of both \cref{eq:verifygoal} and \cref{eq:almostgoal} are identical. Therefore, it suffices to prove that
\begin{equation} \label{eq:goal2}
    \tr\left[(\id - \mat{Y}^2)^{-T} + (\id - \mat{Y}^2)^{-1}\right] = \tr \left[(\id + \mat{Y})^{-T} + (\id - \mat{Y})^{-T}\right] \,,
\end{equation}
from which \cref{eq:verifygoal} will follow.

To prove \cref{eq:goal2}, we note that
\begin{equation}
    \begin{split}
        &\tr \left[(\id + \mat{Y})^{-T} + (\id - \mat{Y})^{-T}\right] \\
        &\quad = \tr \left[ (\id + \mat{Y})^{-T} (\id - \mat{Y})^T (\id - \mat{Y})^{-T} + (\id + \mat{Y})^{-T} (\id + \mat{Y})^T (\id - \mat{Y})^{-T}\right] \\
        &\quad = \tr \left[ (\id + \mat{Y})^{-T}(\id - \mat{Y} + \id + \mat{Y})^T (\id - \mat{Y})^{-T} \right] \\
        &\quad = 2\tr \left[ (\id + \mat{Y})^{-T} (\id - \mat{Y})^{-T}\right] \\
        &\quad = 2 \tr \left[ (\id - \mat{Y}^2)^{-T} \right] \\
        &\quad = \tr \left[ (\id - \mat{Y}^2)^{-T} + (\id - \mat{Y}^2)^{-1} \right] \,.
    \end{split}
\end{equation}
This proves the nonnegativity of the optimal shift factor.
\end{proof}

We remark that the critical step that requires isotropy is the assertion that the left hand sides of both \cref{eq:verifygoal} and \cref{eq:almostgoal} are equal, namely that
\begin{equation}
    \tr \left[ (\id - \mat{Y})^{-1} (\id - \mat{Y})^{-T}\right] = \tr \left[ (\id - \mat{Y})^{-T} (\id - \mat{Y})^{-1} \right] \,.
\end{equation}
This is no longer necessarily true if we replace the isotropic trace operator $\tr(\cdot)$ with an anisotropic operator $\tr(\autoCor^{1/2} (\cdot) \autoCor^{1/2})$ for general $\autoCor$. 

We also note that the theorem above only proves $\optAugFac \geq 0$ and not $\optAugFac > 0$. Getting to $\optAugFac > 0$ requires an additional condition:

\begin{corollary} \label{coro:first}
    Under the conditions of \cref{thm:thm}, the optimal shift factor is positive if and only if $\Pb((\sampMat \truMat^{-1} - \id)^T \neq -(\sampMat \truMat^{-1} - \id)) > 0$.
\end{corollary}

\begin{proof}
Note that the sole inequality in \cref{thm:thm} is
\begin{equation} \label{eq:ineqprob}
    \E \tr \left[ ((\id + \transErrMat)^{-T} - (\id - \transErrMat)^{-1})((\id + \transErrMat)^{-1} - (\id - \transErrMat)^{-T})\right] \geq 0 \,.
\end{equation}
It thus suffices to show the equivalence between the two events
\begin{equation} \label{eq:event1}
    \left\{ \tr \left[ ((\id + \transErrMat)^{-T} - (\id - \transErrMat)^{-1})((\id + \transErrMat)^{-1} - (\id - \transErrMat)^{-T})\right] = 0 \right\} \,,
\end{equation}
and
\begin{equation}
    \left\{ (\sampMat \truMat^{-1} - \id)^T = -(\sampMat \truMat^{-1} - \id)) \right\} = \left\{ \transErrMat^T = - \transErrMat \right\} \,.
\end{equation}
    Clearly, if $\transErrMat$ is anti symmetric, then this implies \cref{eq:event1}. Conversely, if \cref{eq:event1} holds, then because the matrix inside the trace is positive semi-definite, the only way that the trace can be zero is if 
\begin{equation}
    (\id + \transErrMat)^{-T} - (\id - \transErrMat)^{-1} = \mat{0} \,.
\end{equation}
Rearranging the above gives $\transErrMat^T = -\transErrMat$. Therefore \cref{eq:ineqprob} is strictly positive in expectation if and only if $\Pb((\sampMat \truMat^{-1} - \id)^T = -(\sampMat \truMat^{-1} - \id))) > 0$, proving the corollary.
\end{proof}

It is interesting to note that juxtaposition of the above result with the results for SPD matrices from Etter, Ying '20 \cite{etter2020operator}. For SPD matrices, the optimal shift factor is always positive unless $\sampMat = \truMat$ almost surely. On the other hand, \cref{coro:first} tells us that for general matrices, the further $\sampMat \truMat^{-1} - \id$ is on average from its negative transpose $-(\sampMat \truMat^{-1} - \id)^T$, the larger the gap in \cref{eq:ineqprob}, and hence the larger $\optAugFac$ and better operator augmentation will perform. In the worst case scenario, it is possible to adversarially force $\sampMat \truMat^{-1} - \id$ to be anti-symmetric almost surely, in which case, operator augmentation will be no better than the naive estimate. 

\section{The Necessity of Conditions} \label{sec:assumptions}

In this section, we provide adversarial examples of how violating the conditions in \cref{sec:conditions} can lead to situations where the optimal shift factor $\optAugFac$ for $\augMat = \sampMat^{-1}$ is negative. This provides a concrete lens of how the group $\gl$ differs from the symmetric positive definite cone $\pdcone$, as such situations are not possible in the SPD (symmetric positive definite) case. These examples demonstrate how one might use the structure of $\gl$ to construct situations where increasing the ``noise'' in our estimator can actually lead to more accurate results.

\subsection{The Necessity of Isotropy} \label{sec:isotropy}

First, we investigate the necessity of the isotropy condition. Suppose that $\truMat = \id$ and $\errMat$ has two-atom distribution with atoms:
\begin{equation}
    \errAtom_1 = \left[\begin{array}{cc} 0 & k \\ 0 & -k \end{array}\right], \qquad \errAtom_2 = \left[\begin{array}{cc} 0 & -k \\ 0 & k \end{array}\right] \,,
\end{equation}
where each atom occurs with equal probability, and $k \gg 1$. Clearly in this situation, the error distribution of $\errMat$ is symmetric and has mean zero. Therefore, all of the other conditions in \cref{sec:conditions} are met. The shifted operator $\augOpAt{\augFac}^{-1}$ is given by $\augOpAt{\augFac}^{-1} = (1 - \augFac) \sampMat^{-1}$ and the random matrix $\sampMat$ has two outcomes with equal probability,
\begin{equation}
    \truMat_1 = \left[\begin{array}{cc} 1 & k \\ 0 & -k + 1 \end{array}\right], \qquad \truMat_2 = \left[\begin{array}{cc} 1 & -k \\ 0 & k + 1 \end{array}\right] \,.
\end{equation}
These outcomes have inverses
\begin{equation}
    \truMat_1^{-1} = \left[\begin{array}{cc} 1 & k / (k + 1) \\ 0 & - 1/(k - 1) \end{array}\right], \qquad \truMat_2^{-1} = \left[\begin{array}{cc} 1 & k / (k + 1) \\ 0 & 1 / (k - 1) \end{array}\right] \,.
\end{equation}
This means that in either case, we have
\begin{equation}
   \sampMat^{-1} = \left[\begin{array}{cc} 1 & 1 \\ 0 & 0 \end{array}\right] + O(1/k) \,.
\end{equation}

With this matrix ensemble, we can take the prior $\prior$ to be deterministic, such that
\begin{equation}
    \rhs = \left[\begin{array}{c} 2 \\ - 1 \end{array}\right] \,.
\end{equation}
This immediately makes the problem with this setup evident, as
\begin{equation}
    \sampMat^{-1} \rhs = \left[\begin{array}{c} 1 \\ 0 \end{array}\right] + O(1/k), \qquad \truMat^{-1} \rhs = \left[\begin{array}{c} 2 \\ -1 \end{array}\right] \,.
\end{equation}
It is therefore clear that the objective
\begin{equation}
    \errBayes{\augOpAt{\augFac}}{\truMat^T \truMat} = \errBayes{\augOpAt{\augFac}}{\id} = \E \| (1 - \augFac) \sampMat^{-1} \rhs - \truMat^{-1}\rhs \|_2^2
\end{equation}
must achieve its minimum at
\begin{equation}
    \optAugFac = -1 + O(1/k) \,.
\end{equation}

One might initially conclude that the ability of this example to undermine \cref{thm:thm} has something to do with the assumption that $k$ is very large. This is not the case. The largeness of $k$ is assumed only for illustrative purposes. One can verify numerically (see \cref{fig:lowesteig}) for the above choice of $\sampMat$, that $\E[\sampMat^{-T} \sampMat^{-1} - \sampMat^{-T}]$ has a negative eigenvalue for all $k \neq \pm 1, 0$. This means that if we let $vc{v}$ be the corresponding eigenvector and take $\autoCor = \vc{v} \vc{v}^T$, the quantity determining the sign of $\optAugFac$ (see \cref{eq:neededcond}) becomes
\begin{equation}
    \E \langle \sampMat^{-1} , \sampMat^{-1} - \truMat^{-1} \rangle_{\truMat^T \truMat, \autoCor} = \vc{v}^T \E[\sampMat^{-T} \sampMat^{-1} - \sampMat^{-T}] \vc{v} < 0 \,.
\end{equation} 
Therefore, for any $k \neq \pm 1, 0$, one can find a $\rhs$ such that $\optAugFac$ is negative.

\begin{figure}
    \centering
    \includegraphics[scale=0.7]{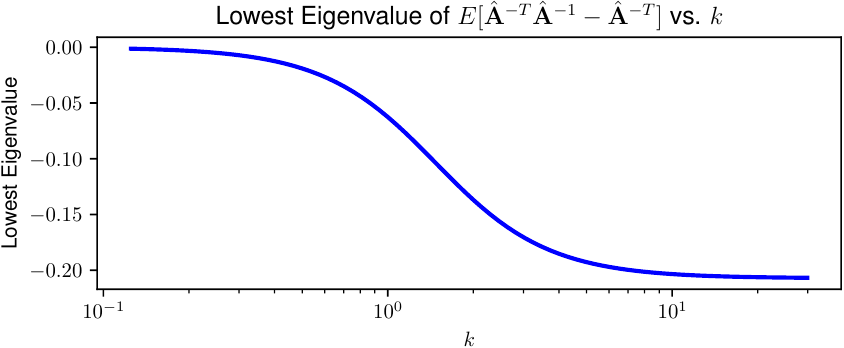}
    \caption{The lowest eigenvalue of $\E[\sampMat^{-T} \sampMat^{-1} - \sampMat^{-T}]$ as defined in \cref{sec:isotropy}.}
    \label{fig:lowesteig}
\end{figure}

In particular, one should note that without a requirement of positive definiteness, it is possible to create a situation where $\sampMat$ is always ``larger'' than $\truMat$. This runs counter to the intuition behind operator shifting in the symmetric positive definite case \cite{etter2020operator}, where such a situation is not possible. 

\subsection{Outlier Masking and the Importance of Noise Symmetry} \label{sec:noisesym}

The second condition that one needed to prove the results in \cref{sec:gaurantees} is the presence of symmetry in the noise distribution. In this section, we will see how if this is not required, it is possible to construct an adversarial example where $\optAugFac < 0$ for the shift \cref{eq:basicaug}. For this example, we take $\truMat = \id$ and let $\errMat$ have a distribution of three equally probable atoms, given by
\begin{equation}
    \errAtom_1 = \id, \qquad \errAtom_2 = k \id, \qquad \errAtom_3 = -(k + 1)\id  \,,
\end{equation}
where $k \gg 1$. Note that this distribution is mean zero. Computing the atoms of the distribution $\sampMat^{-1}$,
\begin{equation}
    \truMat^{-1}_1 = \frac{1}{2} \id, \qquad \truMat^{-1}_2 = \frac{1}{k + 1} \id, \qquad \truMat^{-1}_3 = -\frac{1}{k} \id \,.
\end{equation}
Now, note that in order to minimize the quantity
\begin{equation}
    \errBayes{\augOpAt{\augFac}}{\truMat^T \truMat} = \E \| \augOpAt{\augFac}^{-1} - \id \|_F^2 = \frac{2}{3} \left(\frac{1 - \augFac}{2} - 1\right)^2 + \frac{2}{3} \left(\frac{1 - \augFac}{k + 1} - 1\right)^2 + \frac{2}{3} \left(-\frac{1 - \augFac}{k} - 1\right)^2 \,,
\end{equation}
one can verify by taking the derivative and setting it to zero that
\begin{equation}
    \optAugFac = -1 + O(1/k) \,.
\end{equation}
Therefore, the optimal shift will grow the inverse operator instead of shrinking it. 

The message of this example is that \emph{outliers in the matrix noise can mask distribution imbalances in the region near $\truMat$} that can cause $\E[\sampMat^{-1}]$ to both lie in the direction of $\truMat^{-1}$ while at the same time being dominated by $\truMat^{-1}$. Indeed, we have that $\E[\sampMat^{-1}] \approx \frac{1}{2} \id \preceq \id = \truMat^{-1}$ (it bears repeating that such a feat is impossible in the symmetric positive definite setting where $\E[\sampMat^{-1}] \succeq \truMat^{-1}$ \cite{etter2020operator}). The importance of noise symmetry is that it forces the distribution of $\errMat$ to be balanced in the region around $\truMat$, even if the distribution contains large outliers.

\subsection{The Importance of Conditioning and a Counter-Example for the Frobenius Norm} \label{sec:condition_frob}

Our final counter-example concerns the use of the Frobenius norm in the objective rather than the residual norm. In the symmetric positive definite case, one can prove the positivity of the optimal shift factor for a large range of different objective norms \cite{etter2020operator}. However, as we will see in this section, there are adversarial ways to break norms other than the residual norm in the nonsymmetric case. We will focus here on giving an example that shows how if using the Frobenius norm instead of the residual norm makes it possible to have a negative optimal shift factor for the shift \cref{eq:basicaug}.

To begin, consider the ground truth matrix
\begin{equation}
    \truMat = \left[\begin{array}{cc}
        1 & - \epsilon \\
        \epsilon & 0 
    \end{array}\right] \,,
\end{equation}
where $0 < \epsilon \ll 1$. For the noise $\errMat$, we reuse the noise distribution from \cref{sec:isotropy} --- consider a two-atom distribution with atoms
\begin{equation}
    \errAtom_1 = \left[\begin{array}{cc} 0 & 1 \\ -1 & 0 \end{array}\right], \qquad \errAtom_2 = \left[\begin{array}{cc} 0 & -1 \\ 1 & 0 \end{array}\right] \,,
\end{equation}
where each atom has equal probability. Note that this distribution is symmetric and mean zero. Since $\epsilon$ is extremely small, this means that the distribution of $\sampMat$ will have two atoms whose inverses are approximately
\begin{equation}
    \truMat^{-1}_1 \approx \left[\begin{array}{cc} 1 & 1 \\ -1 & 0 \end{array}\right]^{-1} = \left[\begin{array}{cc} 0 & 1 \\ -1 & 1 \end{array}\right], \qquad \truMat^{-1}_2 \approx \left[\begin{array}{cc} 1 & -1 \\ 1 & 0 \end{array}\right]^{-1} = \left[\begin{array}{cc} 0 & -1 \\ 1 & 1 \end{array}\right] \,.
\end{equation}
Note that the inverses of these atoms have the same Frobenius norm as the atoms themselves. In sharp contrast, the ill-conditioning of $\truMat$ means that $\truMat^{-1}$ is an order of magnitude larger than $\truMat$,
\begin{equation}
    \truMat^{-1} = \left[\begin{array}{cc}
        0 & -\epsilon^{-1}  \\
        \epsilon^{-1} & \epsilon^{-2} 
    \end{array}\right] \,.
\end{equation}
Immediately, we see that in order for $\optAugFac$ to minimize
\begin{equation}
    \errBayes{\augOpAt{\augFac}}{\id} = \frac{1}{2} \| (1 - \augFac) \truMat^{-1}_1 - \truMat^{-1}\|_F^2 + \frac{1}{2} \| (1 - \augFac) \truMat^{-1}_2 - \truMat^{-1} \|_F^2  \,,
\end{equation}
one must therefore have $\optAugFac \sim -\epsilon^{-2}$, for which $\errBayes{\augOpAt{\optAugFac}}{\id} \sim \epsilon^{-1}$. For other growth orders of $\augFac$, one has $\errBayes{\augOpAt{\augFac}}{\id} = \omega(\epsilon^{-1})$. It is therefore clear that for $\epsilon$ small enough, $\optAugFac$ will be negative. 

In contrast to the $L_2$ norm, note that the residual norm matrix for this problem is given by
\begin{equation}
    \truMat^T \truMat = \left[\begin{array}{cc}
        1 + \epsilon^2 & \epsilon  \\
        \epsilon & 0 
    \end{array}\right] \,,
\end{equation}
Therefore, we see that the reason why \cref{thm:thm} holds in the residual norm but not in the $L_2$ norm is because the residual norm places significantly less weight on the part of the matrix $\truMat^{-1}$ that contributes to the $\truMat^{-1}$'s large increase of magnitude over $\truMat$. This means that the residual norm $\truMat^T \truMat$ accounts for ill-conditioning on $\truMat$ in a way that the $L_2$ norm does not.

This example therefore also demonstrates the importance of conditioning in the ground truth matrix $\truMat$. It is perfectly possible that $\truMat$ lies close to a singular matrix, while the outcomes of $\sampMat$ are moved away from singularity by the noise imparted by $\errMat$. If this is the case, $\E[\sampMat^{-1}]$ will be small in magnitude compared to $\truMat^{-1}$ and shrinking the operator $\sampMat^{-1}$ further will not reduce average error in the Frobenius sense.

Note that symmetric positive definite setting \cite{etter2020operator} avoids this issue, since if $\sampMat$ is symmetric positive definite everywhere, it is impossible for $\E[\sampMat]$ to be close to the origin without a significant chunk of the probability distribution also lying close to the origin. This ensures that $\E[\sampMat^{-1}]$ will always spectrally dominate $\sampMat^{-1}$, and shifting the operator $\sampMat^{-1}$ towards $\mat{0}$ will reduce error. 

\section{The Small Noise Regime Allows for Nonsymmetric Noise} \label{sec:small_noise}

As we saw in \cref{sec:noisesym}, the presence of large adversarial outliers can completely mask local imbalances in the noise distribution near $\truMat$. These local imbalances can be severe enough to invalidate \cref{thm:thm}. Naturally, the example presented in \cref{sec:noisesym} is quite extreme, so one might ask if the issue inherent is not necessarily the \emph{symmetry} of the noise, but rather the \emph{magnitude} of the noise. To answer this question, we consider the \emph{small noise regime}, where deviations in $\transErrMat$ are very small relative to $\E[\transErrMat] = \id$. It turns out, that if one assumes that terms of order $O(\E[\|\transErrMat\|_F^3])$ are negligible --- what we term the \emph{small noise regime} --- then the symmetry assumption is unnecessary, as we will see momentarily.

For this discussion, we duplicate the setting of \cref{thm:thm}, except we replace the condition that the noise distribution is symmetric with the condition that noise terms of the order $O(\E[\|\transErrMat\|_F^3])$ are negligible. Recall that the statement necessary for \cref{thm:thm} to be true was
\begin{equation} \label{eq:smallgoal}
\begin{split}
     \E \langle \sampMat^{-1} , \sampMat^{-1} - \truMat^{-1} \rangle_{\truMat^T \truMat, \id} &= \tr \E \left[(\id + \transErrMat)^{-T} (\id + \transErrMat)^{-1} - (\id + \transErrMat)^{-T}\right] \geq 0 \,. \\
\end{split}
\end{equation}
We define the function $f$,
\begin{equation}
    f(\mat{Y}) \equiv \tr \left[(\id + \mat{Y}))^{-T} (\id + \mat{Y}))^{-1} - (\id + \mat{Y}))^{-T}\right] \,.
\end{equation}
Taking a second order Taylor expansion of $f$ about $\transErrMat = \mat{0}$, we obtain
\begin{equation}
    f(\transErrMat) = f(\mat{0}) + \delta f(0) \, \transErrMat + \frac{1}{2} \delta^2 f(0) \, (\transErrMat, \transErrMat) + O(\|\transErrMat\|_F^3) \,.
\end{equation}
Note that $f(\mat{0}) = 0$ and that the first term is linear in $\transErrMat$. Hence, in expectation, both of these terms vanish and we are left with
\begin{equation}
    \E[ f(\transErrMat)] = \frac{1}{2} \E[\delta^2 f(0) \, (\transErrMat, \transErrMat)] + O(\E[\|\transErrMat\|_F^3]) \,.
\end{equation}
A calculation of $\delta^2 f(0) \, (\transErrMat, \transErrMat)$ gives
\begin{equation}
\begin{split}
    \E[\delta^2 f(0) \, (\transErrMat, \transErrMat)] &= \E\tr \left[ 2 \transErrMat^T \transErrMat^T + 2 \transErrMat^T \transErrMat + 2 \transErrMat \transErrMat - 2 \transErrMat^T \transErrMat^T \right] \\
    &= 2 \, \E \tr \left[\transErrMat^T \transErrMat + \transErrMat \transErrMat \right] \\
    &= \E\tr \left[ \transErrMat^T \transErrMat^T + \transErrMat^T \transErrMat + \transErrMat \transErrMat^T + \transErrMat \transErrMat \right] \\
    &= \E\tr \left[ (\transErrMat + \transErrMat^T)^T (\transErrMat + \transErrMat^T) \right] \geq 0 \,.
\end{split}
\end{equation}
Indeed, if $\transErrMat$ is not anti-symmetric almost surely, then the above is in fact a strict inequality. This mirrors the result of \cref{coro:first}. Regardless, we know that \cref{eq:smallgoal} must be true to second order.

\section{Main Theorem for More General Shifts} \label{sec:aniso}

As mentioned in the introduction, in the elliptic case \cite{etter2020operator}, one can prove a reduction in error for a variety of different shifts. In particular, the energy norm in the elliptic shifting setting has an extensive theory regarding the approximation of $\optAugFac$. This therefore begs the question: do more general shifts of the form \cref{eq:gen_aug} retain their nice properties in the nonsymmetric shifting setting --- where the residual norm plays the role of the energy norm in the elliptic shifting setting? \Cref{sec:isotropy} provides a definitive answer to this question: it is not possible unless the modified second moment is isotropic and the chosen norm is the residual norm.
 
Nonetheless, while the wide number of choices regarding norms and shifts do not translate to the nonsymmetric shifting setting, the operator shifting framework does provide a way for handling anisotropic $\autoCor \neq \id$. Namely, one chooses the operator shift
\begin{equation} \label{eq:autocoraug}
    \augMat \equiv \sampMat^{-1} \autoCor^{-1} \,.
\end{equation}
It is immediate that the results of \cref{thm:thm} hold for this choice of shift ---  as the $\autoCor^{-1}$ will cancel the $\autoCor$ in the error objective. We restate this conclusion into the main theorem of this paper:

\begin{theorem} \label{thm:main}
(Main Theorem) Let $\sampMat$ be a random matrix, invertible almost everywhere, such that $\E[\sampMat^{-2}]$ exists. Suppose that the distribution of $\sampMat$ is symmetric about its mean, and that $\E[\sampMat] = \truMat$. Let $\rhs \sim \prior$ with nonsingular second moment matrix $\E[\rhs \rhs^T] = \autoCor$. Consider the residual error
\begin{equation} \label{eq:general_obj}
    \errBayes{\augOpAt{\augFac}}{\truMat^T \truMat} = \E_{\sampMat} \E_{\rhs \sim \prior} \left[\|\augOpAt{\augFac}^{-1} \rhs - \truMat^{-1} \rhs \|^2_{\truMat^T \truMat}\right] = \E \left[\|\augOpAt{\augFac}^{-1} - \truMat^{-1} \|^2_{\truMat^T \truMat, \autoCor}\right]\,.
\end{equation}
Consider the operator shift:
\begin{equation}
    \augOpAt{\augFac}^{-1} = \sampMat^{-1} - \augFac \sampMat^{-1} \autoCor^{-1} \,.
\end{equation}
Then the $\augFac$ optimizing \cref{eq:general_obj} is always nonnegative. Furthermore, the optimal $\augFac$ is strictly positive if and only if $\Pb((\sampMat \truMat^{-1} - \id)^T \neq -(\sampMat \truMat^{-1} - \id)) > 0$.
\end{theorem}

\begin{proof}
As in \cref{thm:thm}, we must verify
\begin{equation} \label{eq:thm2cond}
    \E \langle \sampMat^{-1} \autoCor^{-1}, \sampMat^{-1} - \truMat^{-1} \rangle_{\truMat^T \truMat, \autoCor} \geq 0 \,.
\end{equation}
But note, by definition in \cref{eq:normdef} and symmetry of $\autoCor$, that
\begin{equation}
    \langle \sampMat^{-1} \autoCor^{-1}, \sampMat^{-1} - \truMat^{-1} \rangle_{\truMat^T \truMat, \autoCor} = \langle \sampMat^{-1}, \sampMat^{-1} - \truMat^{-1} \rangle_{\truMat^T \truMat, \id}
\end{equation}
The condition \cref{eq:thm2cond} is therefore equivalent to
\begin{equation}
    \E \langle \sampMat^{-1}, \sampMat^{-1} - \truMat^{-1} \rangle_{\truMat^T \truMat, \id} \geq 0 \,,
\end{equation}
which we proved in \cref{thm:thm}. Likewise, the claim about positivity was proved in \cref{coro:first}.
\end{proof}

\section{Bootstrap Formalism} \label{sec:bootstrap}

In this section, we prepare to give an algorithm for estimating $\optAugFac$ by providing a mathematical formalism that will enable us to write down an algorithm using bootstrap Monte Carlo. In our formalism, we assume that there exists some underlying parameter space $\paramSpace$ (with a sigma algebra $\Sigma$) that produces matrices $\mat{M} = \paramToMat(\params)$ through a measurable mapping $\paramToMat : \paramSpace \longrightarrow \gl$. Here $\gl$ denotes the group of non-singular matrices in $\Rdimdim$. Elements $\params \in \paramSpace$ may represent any number of things, e.g., measurements of a scattering background, edge weights, vertex positions, etc, that cary sufficient information to generate their respective matrices $\mat{M} = \paramToMat(\params)$. For example, $\params \in \paramSpace$ may be a weighted graph, and $\paramToMat(\params)$ may denote its Laplacian. 

In this parameter space $\paramSpace$, we assume that there exists some \emph{unobserved} true system parameters $\truParam \in \paramSpace$ that produce the true matrix $\truMat = \paramToMat(\truParam)$. The central assumption of our bootstrap procedure is the ability to sample $\sampMat$ given the parameters $\truParam$ --- much in the same way one can use sufficient statistics to bootstrap samples from a distribution.

To codify this, there must exist a parameterized family of distributions $\matDistAt{\params}$ over $\gl$ that describes the observed randomness in the system if $\params$ were to be the true system parameters. In practice, the following algorithms tend to be easier to implement if one takes the view that noise acts on parameter space $\paramSpace$ rather than directly on the matrices themselves. For most problems, this makes intrinsic sense, i.e., there may be noise in the edge weights in a graph, noise in measured scattering background, etc. Therefore, we assume the existence of parameterized family of distributions $\paramDist{\params}$ on $\paramSpace$ whose pushforward under $\paramToMat$ gives $\matDistAt{\params}$. That is, we use $\paramDist{\params}$ to talk about the noise as distributed over parameters, whereas we use $\matDistAt{\params}$ to talk about noise as distributed over the corresponding matrices. A concrete example of this formalism is given in \cref{sec:badexample}.

\section{Monte Carlo Estimation} \label{sec:montecarl}

To build an algorithm from this theory, we attempt to estimate the optimal shift factor with bootstrap Monte Carlo. The quantity of interest is
\begin{equation} \label{eq:bootbase}
    \optAugFac = \frac{\E_{\matDist} \langle \augMat, \sampMat^{-1} - \truMat^{-1} \rangle_{\truMat^T \truMat, \autoCor}}{\E_{\matDist} \| \augMat\|^2_{\truMat^T \truMat, \autoCor}} \,.
\end{equation}
As mentioned previously, it is not possible to compute this quantity directly, since we only have a single sample $\sampMat$ and we do not have access to the ground truth $\truMat$ or its distribution $\matDist$. We must therefore try to approximate these quantities as best as possible with the available information. Suppose that $\sampParam \in \paramSpace$ is the parameter instance that generates the matrix $\sampMat = \paramToMat(\sampParam)$. To bootstrap \cref{eq:bootbase}, we replace all instances of $\truMat$ with $\sampMat$ and draw random instances from $\matDistAt{\sampParam}$ instead of $\matDist$. We get
\begin{equation}
    \tilde{\augFac} = \frac{\E_{\matDistAt{\sampParam}} \langle \augMat_b, \sampMat^{-1}_b - \sampMat^{-1} \rangle_{\sampMat^T \sampMat, \autoCor}}{\E_{\matDistAt{\sampParam}} \| \augMat_b \|^2_{\sampMat^T \sampMat, \autoCor}} \,,
\end{equation}
where $\augMat_b$ and $\sampMat_b$ are drawn from $\matDistAt{\sampParam}$. This can then be discretized with Monte Carlo:
\begin{equation}
    \hat{\augFac} = \frac{\sum_{i = 1}^m \rhs_i^T \augMat_{b, i}^{T} \sampMat^T \sampMat (\sampMat^{-1}_{b, i} - \sampMat^{-1}) \rhs_i}{\sum_{i = 1}^m \rhs_i^T \augMat_{b, i}^{T} \sampMat^T \sampMat \augMat_{b, i} \rhs_i} \,, \qquad \rhs_i \sim \prior \; \text{i.i.d.},\quad \sampMat_{b, i} \sim \matDistAt{\sampParam} \; \text{i.i.d.} \,.
\end{equation}
If one opts to use the shift provided by \cref{eq:autocoraug}, the expression simplifies to 
\begin{equation} \label{eq:montecarlo}
\begin{aligned}[c]
\hat{\augFac} &= 1 - \frac{\sum_{i = 1}^m \rhs_i^T \sampMat_{b, i}^{-T} \sampMat^T \rhs_i}{\sum_{i = 1}^m \vc{q}_i^T \sampMat_{b, i}^{-T} \sampMat^T \sampMat \sampMat_{b, i}^{-1} \vc{q}_i} \,,
\end{aligned}
\qquad
\begin{aligned}[c]
\rhs_i &\sim \mathcal{N}(\mat{0}, \id) \; \text{i.i.d.}, \\
\vc{q}_i &\sim \mathcal{N}(\mat{0}, \autoCor^{-1}) \; \text{i.i.d.}, \\
\sampMat_{b, i} &\sim \matDistAt{\sampParam} \; \text{i.i.d.} \,.
\end{aligned}
\end{equation}
Note that $\rhs_i$ and $\vc{q}_i$ do not necessarily have to be normal -- they must just have the same second moment matrix as the normal distributions specified. For good measure, one might also threshold the above expression to guarantee $\hat{\augFac} \geq 0$. Finally, one can now estimate the true solution to the system \cref{eq:base_system} via
\begin{equation} 
    \impSolAt{\hat{\augFac}} = \augOpAt{\hat{\augFac}}^{-1} \rhs = (\sampMat^{-1} - \hat{\augFac} \sampMat^{-1} \autoCor^{-1}) \rhs = \sampMat^{-1} (\rhs - \hat{\augFac} \autoCor^{-1} \rhs) \,.
\end{equation}
The figure \cref{fig:my_label} presents a flow-chart of the process.

\begin{figure}
    \centering
    \includegraphics[scale=0.45]{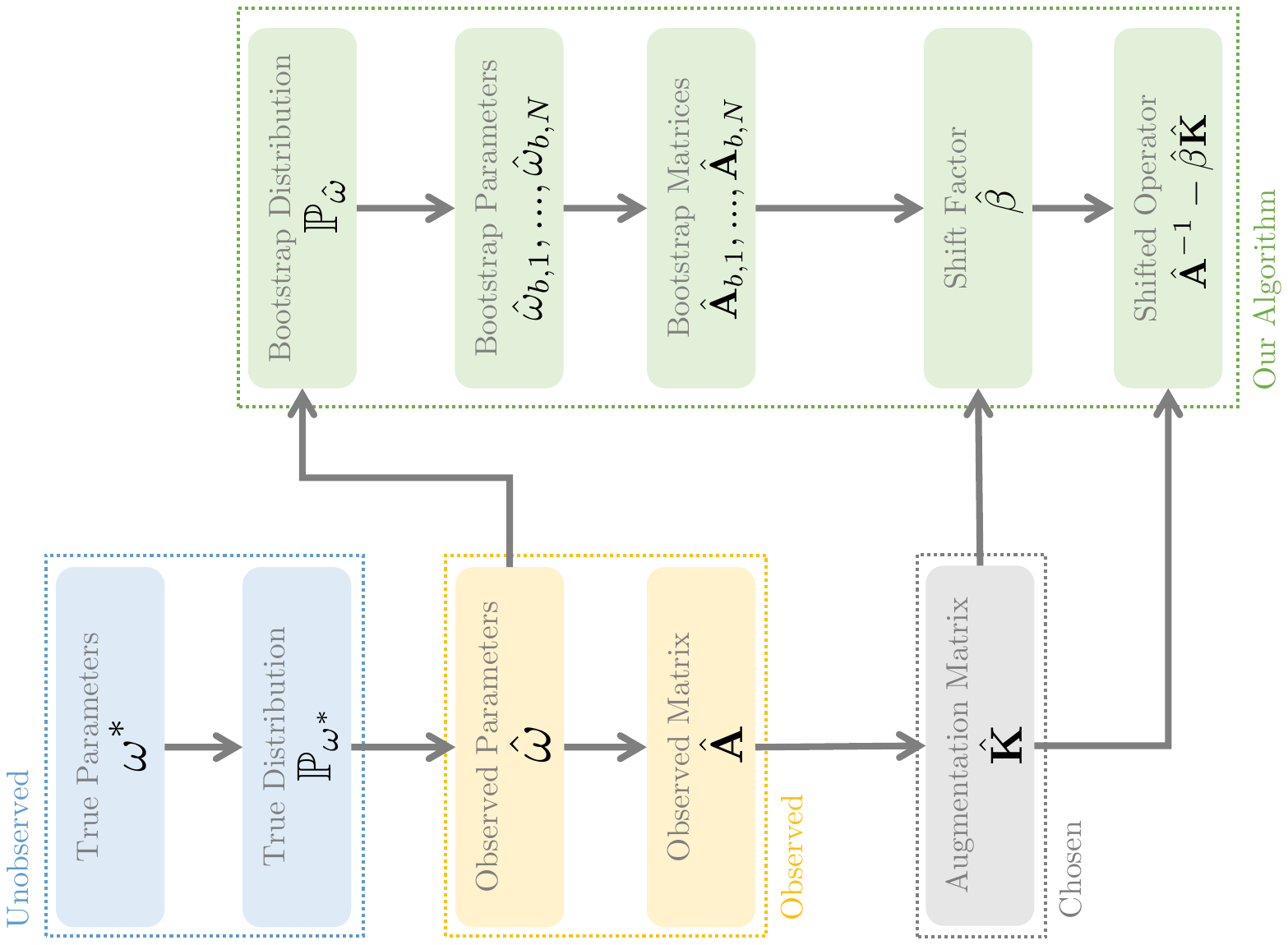}
    \caption{A flow-chart of the probabilistic setting of operator shifting as well as the algorithm itself. Operator shifting aims to find a $\augFac$ that gives an optimal reduction in error given a shift matrix $\augMat$.}
    \label{fig:my_label}
\end{figure}

\subsection{Approximation via Taylor Expansion}

Note that every Monte Carlo sample in \cref{eq:montecarlo} requires inverting a matrix system. There are times where this may be too computationally expensive to be feasible. However, if we are in the \emph{small noise regime}, one may take a Taylor expansion of $\sampMat^{-1}$ about $\truMat^{-1}$ --- as this means one only has the factorize an operator once for the whole estimation process. The expansion to second order is given by
\begin{equation}
    \sampMat^{-1} = \truMat^{-1} - \truMat^{-1} \errMat \truMat^{-1} + 2 \truMat^{-1} \errMat \truMat^{-1} \errMat \truMat^{-1} + O(\errMat^3) \,. 
\end{equation}
Inserting this expression into \cref{eq:montecarlo} yields
\begin{equation} \label{eq:taylor}
\hat{\augFac} \approx 1 - \frac{\sum_{i = 1}^m \rhs_i^T[ \id + 2 (\errMat_{b, i} \sampMat^{-1})^2 ]\rhs_i}{\sum_{i = 1}^m \vc{q}_i^T [\id + (\errMat_{b, i} \sampMat^{-1})^T (\errMat_{b, i} \sampMat^{-1}) + 4 (\errMat_{b, i} \sampMat^{-1})^2 ] \vc{q}_i} \,,
\end{equation}
note that we have omitted linear terms, because terms linear in $\errMat$ will be zero in expectation, by virtue of the fact that $\E[\errMat] = 0$. It is possible that higher orders of truncation may produce better results; however, unlike the elliptic shifting setting \cite{etter2020operator}, it is difficult to prove guarantees about the quality of these truncated expansions --- as one cannot use the machinery of positive definite polynomials available in the elliptic case.

\subsection{How Bootstrap Can Fail When $\optAugFac < 0$} \label{sec:badexample}

To offer a concrete example of why $\optAugFac > 0$ is a desirable trait to have when performing operator shifting --- we offer an adversarial example where bootstrapping can be made arbitrarily bad at estimating $\optAugFac$ when $\optAugFac$ is allowed to be negative. Naturally, since there are three primary modes of failure for $\optAugFac \geq 0$ (namely, the cases discussed in the previous section), the example will be a bootstrapped version of \cref{sec:condition_frob}.

To convert the example from \cref{sec:condition_frob} into a bootstrap problem we consider the sample space $\paramSpace = \mathbb{R}$ and let the mapping $\paramToMat$ be given by:
\begin{equation}
    \paramToMat : \params \mapsto \left[ \begin{array}{cc} 
        1 & \params \\
        -\params & 0 
    \end{array} \right]
\end{equation}
For parameter $\params$, the noise distribution is given by:
\begin{equation}
    \mathbb{P}_{\params} = \frac{1}{2} \left[ \delta_{\params - 1} + \delta_{\params + 1} \right] \,,
\end{equation}
where $\delta_{\params}$ is the dirac delta distribution at $\params$. In \cref{fig:bootstrap} below, we see that bootstrapping the optimal shift factor in this example always returns an average shift factor that is positive, telling the algorithm to shrink the inverse operator, whereas the true optimal shift factor is unboundedly negative as one takes $\truParam = \epsilon \rightarrow 0$, meaning that it would have been optimal to enlarge the inverse operator instead.

\begin{figure}[h]
    \centering
    \includegraphics[width=0.3\paperwidth]{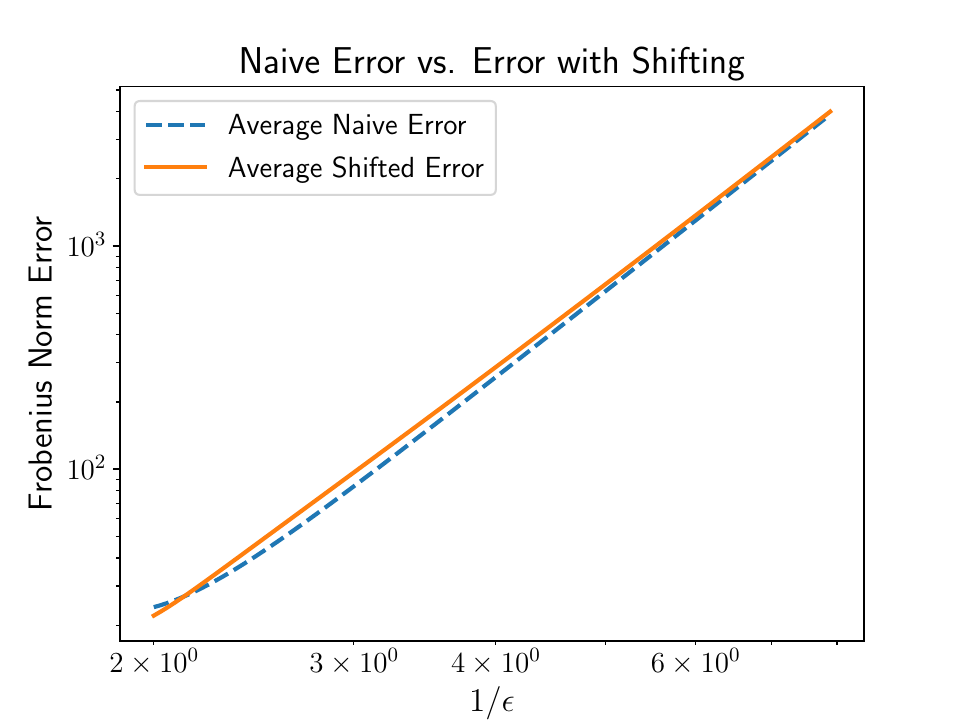}
    \includegraphics[width=0.3\paperwidth]{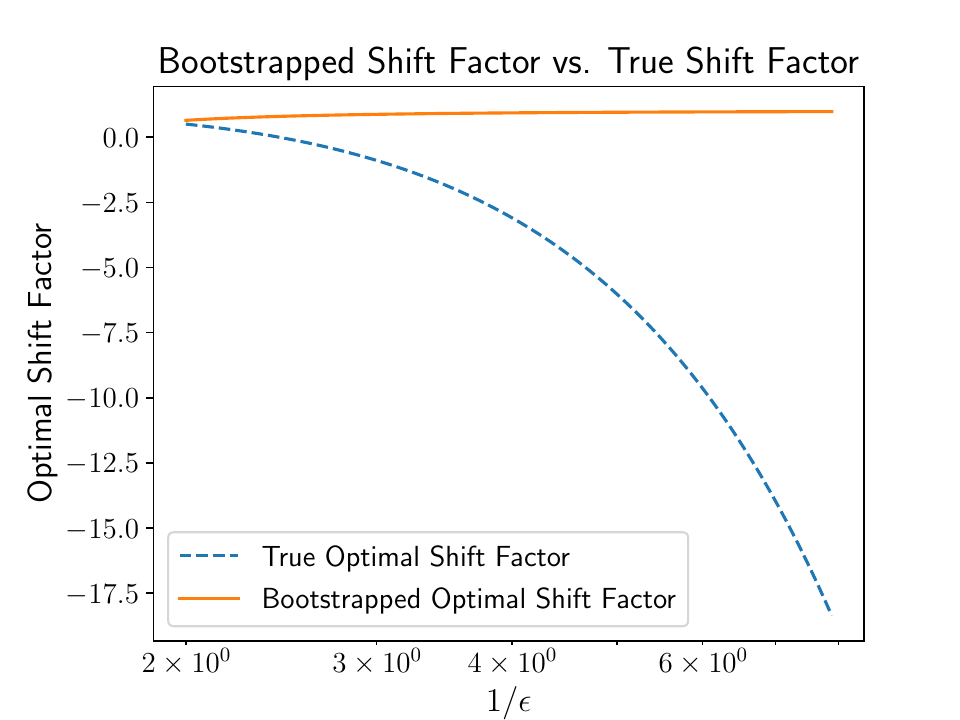}
    \caption{\label{fig:bootstrap} An example where bootstrapping can mis-estimate the optimal shift factor $\optAugFac$ when $\optAugFac < 0$. In particular, we note that the bootstrapped algorithm always returns a positive shift factor on average. On the left: a plot of the error of the naive estimator $\sampMat^{-1}$ versus the error of the estimator $\augOpAt{\bootstrapAugFac}^{-1}$. On the right: the true optimal shift factor $\optAugFac$ versus the average shift factor returned by  bootstrapping. We observe that the error becomes slightly worse when bootstrapped operator shifting is applied because the algorithm shrinks the matrix rather than enlarging it.}
\end{figure}

Notably, the examples in \cref{sec:aniso} and \cref{sec:noisesym} do not exhibit similar failure when bootstrapped. We have tested them both and bootstrapping does quite well in recovering the optimal shift factor of $\optAugFac \approx -1$. We believe that this discrepancy in the effectiveness of bootstrapping has to do with the fact that the above example is substantially worse conditioned.

\section{Numerical Experiments} \label{sec:exp}

To confirm our theoretical results, we will examine the asymmetric operator shift algorithm applied to the problem of more accurately computing a value function for a Markov chain whose probability transition function must be estimated from data. 

\subsection{Background} 
For context, a discrete Markov chain is a stochastic process $X_i$ for $i \in \mathbb{Z}_{\geq 0}$ on a finite state set $V$ is a process that satisfies the Markov property
\begin{equation} \label{eq:markov}
    \mathbb{P}(X_i = v_i \mid X_{i - 1} = v_{i - 1}, ..., X_1 = v_1) = \mathbb{P}(X_i = v_i \mid X_{i - 1} = v_{i - 1}) \,.
\end{equation}
A discrete Markov chain is time homogeneous if the right-hand-side of \cref{eq:markov} does not depend on the time $i$. In this situation, the Markov chain is completely characterized by its probability transition matrix
\begin{equation}
    \mat{P}_{v,u} = \mathbb{P}(X_i = u \mid X_{i - 1} = v) \,,
\end{equation}
as well as a distribution $\mathbb{P}(X_0 = v)$ of the initial state $X_0$. There are countless examples of such Markov chains from the fields of probability, statistics, reinforcement learning, and physics. Good references for the reinforcement learning flavored problems we will introduce shortly include \cite{puterman2014markov, bertsekas2019reinforcement}.

One is often interested in computing or approximating functionals of the process $X_i$. For example, one may think of $X_i$ as an agent navigating through the set $V$ via the means of some fixed policy, where the agent obtains a fixed reward $r(v)$ whenever it transitions from the state $v$. In many situations in reinforcement learning, one is often interested in the average discounted reward $Q(v)$ the agent obtains over its life-cycle when beginning at $X_0 = v$, namely,
\begin{equation} \label{eq:inf}
    Q(v) = \E\left[ \sum_{i = 0}^{\infty} \gamma^i r(X_i) \mid X_0 = v \right] \,.
\end{equation}
$Q(v)$ is typically referred to as the \emph{value function} of $X_i$ with respect to reward function $r(v)$. It is used to gauge the quality of the agent's policy at maximizing the discounted reward it receives over its life-time. The quantity $\gamma \in (0, 1)$ is known as a \emph{discount factor} and determines how much immediate reward is valued against future reward. One notes that the function $Q$ is linear in $r$. One may use first transition analysis to express the relationship between $Q$ and $r$ in matrix form. Isolating the first term in the infinite sum \cref{eq:inf} gives
\begin{equation}
    Q(v) = r(v) + \gamma \sum_{u} \mathbb{P}(X_1 = u \mid X_0 = v) Q(u) \,.
\end{equation}
Alternatively,
\begin{equation}
    \vc{q} = \vc{r} + \gamma \mat{P} \vc{q} \,,
\end{equation}
where $\vc{q}$ is the vector with entries $\vc{q}_v = Q(v)$ and $\vc{r}$ is the vector with entries $\vc{r}_v = r(v)$. This finally produces a linear system for the value function,
\begin{equation}
    \truMat \vc{q} \equiv (\id - \gamma \mat{P}) \vc{q} = \vc{r} \,.
\end{equation}

However, there are many situations where the transition matrix $\mat{P}$ may not be known exactly. In reinforcement learning, for example, one does not have access to $\mat{P}$ itself, but rather a number of finite realizations of the process $X_i$ as it traverses the state space. Therefore, instead of having access to the ground truth $\mat{P}$, one usually has access to a noisy version $\hat{\mat{P}}$ thereof. A naive solve will give
\begin{equation}
    (\id - \gamma \hat{\mat{P}}) \hat{\vc{q}} = \vc{r} \,.
\end{equation}
We will see how operator shifting can be used to reduce the average error between our estimate $\hat{\vc{q}}$ and the ground truth $\vc{q}$ in the residual norm. 

\subsection{Noise Model}

One popular way to approximate a Markov chain's probability transition matrix from a sample $\hat{X}_0, \hat{X}_1, ..., \hat{X}_t$ of the Markov chain is to approximate the probability $\mathbb{P}(X_i = v \mid X_{i - 1} = u)$ by examining the fraction of times $X_i$ transitions to $v$ from $u$ out of the total number of times $X_i$ transitions from $u$, i.e.,
\begin{equation}
    \mathbb{P}(X_i = v \mid X_{i - 1} = u) = \frac{\mathbb{P}(X_i = v \text{ and } X_{i - 1} = u)}{\mathbb{P}(X_{i - 1} = u)} \approx \frac{\# \{ i \mid \hat{X}_i = v \text{ and } \hat{X}_{i - 1} = u\} }{\# \{ i \mid \hat{X}_i = u \text{ and } i \geq 1 \}} \,. 
\end{equation}
But while this estimator is commonly used --- it is difficult to ensure coverage of the state space in a way that is natural, as one must take care to ensure that $\# \{ i \mid \hat{X}_i = u \text{ and } i \geq 1 \}$ is positive for every $u$. 

Therefore, to model the uncertainty in a Markov chain constructed from data, we assume that, instead of observing the first $t$ states of the Markov chain $X_i$, we instead observe the first $N$ transitions of the Markov chain $X_i$ out of every state $v$. We denote the first $N$ transitions out of a state $v$ as $\hat{Y}_{v, 1}, \hat{Y}_{v, 2}, ..., \hat{Y}_{v, N}$. If the Markov chain is irreducible, $\hat{Y}_{v, 1}, \hat{Y}_{v, 2}, ..., \hat{Y}_{v, N}$ are well-defined almost surely, and have distribution
\begin{equation}
    \hat{Y}_{v, 1}, \hat{Y}_{v, 2}, ..., \hat{Y}_{v, N} \sim \mathbb{P}(X_i \mid X_{i - 1} = v)\,, \qquad \text{i.i.d.}
\end{equation}
One may then easily estimate the probability transition matrix by using the empirical distribution of $\mathbb{P}(X_i \mid X_{i - 1} = v)$ to construct each row of $\hat{\mat{P}}$,
\begin{equation}
    \hat{\mat{P}}_{v, u} \equiv \frac{\# \{ \hat{Y}_{v, i} = u \}}{N} \approx \mathbb{P}(X_i = u \mid X_{i - 1} = v) = \mat{P}_{v, u} \,.
\end{equation}
Note that $\hat{\mat{P}}$ is unbiased and hence that
\begin{equation}
    \E[ \sampMat ] = \truMat \,.
\end{equation}
We pose the question: is it possible to use operator shifting to increase the accuracy of the naive estimate of the value function $Q$? Our numerical results suggest that the answer to this question is yes.

\subsection{Numerical Results}

To test the theory, we consider three different Markov chains on 1D, 2D, and 3D grids respectively. For our experiments, we use the operator shift $\augMat = \sampMat^{-1}$,
\begin{equation}
    \tilde{\mat{A}}^{-1} = \sampMat^{-1} - \hat{\augFac} \augMat = (1 - \hat{\augFac}) \sampMat^{-1} = (1 - \hat{\augFac}) (\id - \gamma \hat{\mat{P}})^{-1} \,,
\end{equation}
where we use the Taylor expanded bootstrapped Monte Carlo approximation $\hat{\augFac}$ from \cref{eq:taylor}. We compute our shifted solution via
\begin{equation}
    \tilde{\vc{q}} = \tilde{\mat{A}}^{-1} \vc{r} \,.
\end{equation}

\subsubsection{Random Walk with Drift on a 1D Grid}

For this example, we consider the Markov chain of a lazy random walk with drift on a 1D periodic grid graph. We let the state space $V$ be the set of integers $\{1, 2, ..., K\}$ where $K = 16$ and take the probability transition matrix to be
\begin{equation}
    \mathbb{P}^{(\text{1D})}_{\ell, r}(X_i = u \mid X_{i - 1} = v) = \begin{cases}
    \ell & u = v - 1 \mod K\\
    1 - \ell - r & u = v \mod K \\
    r & u = v + 1 \mod K \\
    0 & \text{o.w.}
    \end{cases}
\end{equation}
We test a handful of different values for $(\ell, r)$ as well as a handful of different values for the sample count $N$ that determines the ``noise'' in the matrix $\hat{\mat{P}}$ (higher $N$ means less noise). For the reward function --- we consider two cases: for the first, we consider deterministic reward function given by
\begin{equation}
    r^{(\text{1D})}(v) = \sin(4 \pi v / K) \,.
\end{equation}
For the second, we consider a random isotropic reward vector prior where $\vc{r}^{(\text{1D})} \sim \mathcal{N}(\vc{0}, \id)$.

In addition to the above collection of transition matrices, we also consider random walks with the ability to skip a vertex in the grid,
\begin{equation}
    \mathbb{P}^{(\text{1D})}_{\ell_1, \ell_2, r_1, r_2}(X_i = u \mid X_{i - 1} = v) = \begin{cases}
    \ell_1 & u = v - 2 \mod K \\
    \ell_2 & u = v - 1 \mod K \\
    1 - \ell_1 - \ell_2 - r_1 - r_2 & u = v \mod K \\
    r_1 & u = v + 1 \mod K \\
    r_2 & u = v + 2 \mod K \\
    0 & \text{o.w.}
    \end{cases}
\end{equation}
as well as transition matrices corresponding to a random walk on a complete graph,
\begin{equation}
    \mathbb{P}^{(\text{1D})}_{\text{complete}}(X_i = u \mid X_{i - 1} = v) = \frac{1}{K} \,.
\end{equation}

For these 1D experiments, we use the discount factor
\begin{equation}
    \gamma^{(1D)} = 0.99 \,.
\end{equation}

We present the results of our numerical experiments for the specified reward vectors in \cref{tab:results} as well as for isotropic reward vectors in \cref{tab:results_iso}.

\subsubsection{Random walk with Drift on 2D and 3D Grids}

Now let us consider the Markov chain of a lazy random walk with non-uniform drift on a 2D periodic grid graph. We let the state space $V$ be the set of tuples $\{1, 2, ..., K\} \times \{1, 2, ..., K\}$ where $K = 16$ and we consider two probability transition matrices, the first a standard uniform random walk:
\begin{equation}
    \mathbb{P}^{(\text{2D})}_{\text{unif}}(X_i = u \mid X_{i - 1} = v) = \begin{cases}
    1/4 & |u - v| = 1 \\
    0 & \text{o.w.}
    \end{cases}
\end{equation}
and the second a random walk with a non-uniform drift:
\begin{equation}
    \mathbb{P}^{(\text{2D})}_{\text{nonunif}}(X_i = u \mid X_{i - 1} = v) = \begin{cases}
    \frac{1}{4} \pm \frac{1}{8} \sin(2 \pi v_x / K)& u = v \mp (1, 0) \\
    \frac{1}{4} \pm \frac{1}{8} \sin(2 \pi v_y / K) & u = v \mp (0, 1) \\
    0 & \text{o.w.}
    \end{cases}
\end{equation}
For the reward function, we consider both a deterministic reward
\begin{equation}
    r^{(\text{2D})}(v) = - \sin(2 \pi v_x / K) \sin (2 \pi v_y / K) \,,
\end{equation}
as well as an isotropic reward vector prior where $\vc{r}^{(\text{2D})} \sim \mathcal{N}(\vc{0}, \id)$. For these 2D experiments, we use the discount factor
\begin{equation}
    \gamma^{(2D)} = 0.99 \,.
\end{equation}

We define similar transition matrices on 3D periodic grid graphs, where $V = \{1, ..., K\}^3$ and $K = 8$:
\begin{equation}
    \mathbb{P}^{(\text{3D})}_{\text{unif}}(X_i = u \mid X_{i - 1} = v) = \begin{cases}
    1/6 & |u - v| = 1 \\
    0 & \text{o.w.}
    \end{cases}
\end{equation}
as well as an analogous 3D random walk with a non-uniform drift:
\begin{equation}
    \mathbb{P}^{(\text{3D})}_{\text{nonunif}}(X_i = u \mid X_{i - 1} = v) = \begin{cases}
    \frac{1}{6} \pm \frac{1}{8} \sin(2 \pi v_x / K)& u = v \mp (1, 0, 0) \\
    \frac{1}{6} \pm \frac{1}{8} \sin(2 \pi v_y / K) & u = v \mp (0, 1, 0) \\
    \frac{1}{6} \pm \frac{1}{8} \sin(2 \pi v_z / K) & u = v \mp (0, 0, 1) \\
    0 & \text{o.w.}
    \end{cases}
\end{equation}
For the reward function, we consider both a deterministic reward
\begin{equation}
    r^{(\text{3D})}(v) =  - \sin(2 \pi v_x / K) \sin (2 \pi v_y / K) \sin(2 \pi v_z / K) \,,
\end{equation}
as well as an isotropic reward vector prior where $\vc{r}^{(\text{3D})} \sim \mathcal{N}(\vc{0}, \id)$. For these 3D experiments, we use the discount factor
\begin{equation}
    \gamma^{(3D)} = 0.9 \,.
\end{equation}

We present the results of our numerical experiments for the specified reward vectors in \cref{tab:results} as well as for isotropic reward vectors in \cref{tab:results_iso}.
 
\subsection{Discussion}

\bgroup
\def\arraystretch{1.2}
\begin{table}[p]
    \centering
\begin{tabular}{l||l|ll|ll}
Chain & Samples ($N$) & Naive Error & $\pm 2 \sigma$ & Shifted Error & $\pm 2\sigma$ \\
\hline \hline
$\mathbb{P}^{(\text{1D})}_{1/4, 1/4}$ & 16 & $166 \%$ & $\pm 13.5 \%$ & $53.5 \%$ & $\pm 1.35 \%$ \\
& 32 & $48.8 \%$ & $\pm 1.50 \%$ & $30.1 \%$ & $\pm 0.53 \%$ \\
& 64 & $20.1 \%$ & $\pm 0.45 \%$ & $16.0 \%$ & $\pm 0.27 \%$ \\
\hline
$\mathbb{P}^{(\text{1D})}_{1/6, 2/6}$ & 16 & $115 \%$ & $\pm 10.2\%$ & $42.2 \%$ & $\pm 1.57 \%$ \\
& 32 & $31.7 \%$ & $\pm 1.11\%$ & $20.8 \%$ & $\pm 0.47 \%$ \\
& 64 & $ 12.7\%$ & $\pm 0.31 \%$ & $10.4 \%$ & $\pm 0.20 \%$ \\
\hline
$\mathbb{P}^{(\text{1D})}_{0, 1/2}$ & 16 & $11.7 \%$ & $\pm 1.06 \%$ & $8.70\%$ & $\pm 0.67\%$ \\
& 32 & $4.02 \%$ & $\pm 0.12\%$ & $3.49\%$ & $\pm 0.04\%$ \\
& 64 & $1.75 \%$ & $\pm 0.04\%$ & $1.63\%$ & $\pm 0.03\%$ \\
\hline
$\mathbb{P}^{(\text{1D})}_{1/8,1/8,1/8,1/8}$ & 16 & $55.2 \%$ & $\pm 2.61 \%$ & $29.9 \%$ & $\pm 0.59 \%$ \\
& 32 & $19.5 \%$ & $\pm 0.44 \%$ & $15.0 \%$ & $\pm 0.24\%$ \\
& 64 & $8.55 \%$ & $\pm 0.15 \%$ & $7.55 \%$ & $\pm 0.12 \%$ \\
\hline
$\mathbb{P}^{(\text{1D})}_{\text{complete}}$ & 16 & $6.63\%$ & $\pm 0.18\%$ & $6.09\%$ & $\pm 0.09\%$ \\
& 32 & $3.25 \%$ & $\pm 0.05\%$ & $3.10 \%$ & $\pm 0.04\%$ \\
& 64 & $1.43 \%$ & $\pm 0.02\%$ & $1.41 \%$ & $\pm 0.02\%$ \\
\hline \hline
$\mathbb{P}^{(\text{2D})}_{\text{unif}}$ & 12 & $206\%$ & $\pm 42.4\%$ & $66.9\%$ & $\pm 3.26\%$ \\
& 24 & $91.3\%$ & $\pm 14.3\%$ & $47.5\%$ & $\pm 3.00\%$ \\
& 48 & $43.3\%$ & $\pm 5.63\%$ & $30.1\%$ & $\pm 2.29\%$ \\
\hline
$\mathbb{P}^{(\text{2D})}_{\text{nonunif}}$ & 12 & $113\%$ & $\pm 20.8\%$ & $52.7\%$ & $\pm 7.91\%$ \\
& 24 & $51.7\%$ & $\pm 7.91\%$ & $33.8\%$ & $\pm 3.08\%$ \\
& 48 & $24.7\%$ & $\pm 3.40\%$ & $19.7\%$ & $\pm 2.05\%$ \\
\hline \hline
$\mathbb{P}^{(\text{3D})}_{\text{unif}}$ & 4 & $92.7\%$ & $\pm 30.9\%$ & $48.3\%$ & $\pm 7.25\%$ \\
& 8 & $38.8\%$ & $\pm 9.09\%$ & $27.8\%$ & $\pm 4.12\%$ \\
& 16 & $18.0\%$ & $\pm 3.66\%$ & $15.2\%$ & $\pm 2.36\%$ \\
\hline
$\mathbb{P}^{(\text{3D})}_{\text{nonunif}}$ & 4 & $77.8\%$ & $\pm 22.6\%$ & $44.3\%$ & $\pm 7.00\%$ \\
& 8 & $33.4\%$ & $\pm 7.86\%$ & $25.1\%$ & $\pm 4.04\%$ \\
& 16 & $15.7\%$ & $\pm 3.20\%$ & $13.5\%$ & $\pm 2.29\%$ \\
\end{tabular} \label{tab:results}
\caption{\textbf{Deterministic Value Function Comparison}: A performance comparison between the accuracy of the shifted estimator $\tilde{\truSol}$ for the solution $\truSol$ versus the naive estimator $\sampSol$. The error is measured as a percentage with respect to the residual norm of the true solution. Since the error is calculated via Monte Carlo, we provide a $95\%$ confidence interval in the $\pm 2\sigma$ column. }
\end{table}
\egroup

\bgroup
\def\arraystretch{1.2}
\begin{table}[p]
    \centering
\begin{tabular}{l||l|ll|ll}
Chain & Samples ($N$) & Naive Error & $\pm 2 \sigma$ & Shifted Error & $\pm 2\sigma$ \\
\hline \hline
$\mathbb{P}^{(\text{1D})}_{1/4, 1/4}$ & 16 & $105 \%$ & $\pm 25.4 \%$ & $48.4 \%$ & $\pm 6.16\%$ \\
& 32 & $30.8 \%$ & $\pm 2.95\%$ & $22.8\%$ & $\pm 1.70\%$ \\
& 64 & $12.7\%$ & $\pm 1.00\%$ & $11.1\%$ & $\pm 0.77\%$ \\
\hline
$\mathbb{P}^{(\text{1D})}_{1/6, 2/6}$ & 16 & $64.1\%$ & $\pm 16.0\%$ & $33.9\%$ & $\pm 5.52\%$ \\
& 32 & $17.2\%$ & $\pm 1.77\%$ & $13.5\%$ & $\pm 1.14\%$ \\
& 64 & $6.84\%$ & $\pm 0.52\%$ & $6.13\%$ & $\pm 0.41\%$ \\
\hline
$\mathbb{P}^{(\text{1D})}_{0, 1/2}$ & 16 & $11.0\%$ & $\pm 3.30\%$ & $8.49\%$ & $\pm 2.40\%$ \\
& 32 & $3.77\%$ & $\pm 0.33\%$ & $3.34\%$ & $\pm 0.26\%$ \\
& 64 & $1.64 \%$ & $\pm 0.11\%$ & $1.54\%$ & $\pm 0.10\%$ \\
\hline
$\mathbb{P}^{(\text{1D})}_{1/8,1/8,1/8,1/8}$ & 16 & $55.1\%$ & $\pm 2.62\%$ & $29.9\%$ & $\pm 0.59\%$ \\
& 32 & $19.5\%$ & $\pm 0.44\%$ & $15.0\%$ & $\pm 0.24\%$ \\
& 64 & $8.55\%$ & $\pm 0.15\%$ & $7.55\%$ & $\pm 0.12\%$ \\
\hline
$\mathbb{P}^{(\text{1D})}_{\text{complete}}$ & 16 & $6.21 \%$ & $\pm 0.29\%$ & $5.81\%$ & $\pm 0.25\%$ \\
& 32 & $2.98\%$ & $\pm 0.13\%$ & $2.89\%$ & $\pm 0.12\%$ \\
& 64 & $1.46\%$ & $\pm 0.06\%$ & $1.43\%$ & $\pm 0.06\%$ \\
\hline \hline
$\mathbb{P}^{(\text{2D})}_{\text{unif}}$ & 12 & $25.7\%$ & $\pm 7.51\%$ & $20.4\%$ & $\pm 4.89\%$ \\
& 24 & $11.3\%$ & $\pm 2.93\%$ & $10.2\%$ & $\pm 2.38\%$ \\
& 48 & $5.35\%$ & $\pm 1.31\%$ & $5.07\%$ & $\pm 1.19\%$ \\
\hline
$\mathbb{P}^{(\text{2D})}_{\text{nonunif}}$ & 12 & $20.2\%$ & $\pm 6.17\%$ & $16.7\%$ & $\pm 4.89\%$ \\
& 24 & $8.96\%$ & $\pm 2.43\%$ & $8.20\%$ & $\pm 2.04\%$ \\
& 48 & $4.26\%$ & $\pm 1.09\%$ & $4.07\%$ & $\pm 1.00\%$ \\
\hline \hline
$\mathbb{P}^{(\text{3D})}_{\text{unif}}$  & 4 & $35.3\%$ & $\pm 35.0\% $ & $26.2\%$ & $\pm 20.8\%$ \\
& 8 & $14.6 \%$ & $\pm 11.8\%$ & $12.7\%$ & $\pm 9.11\%$ \\
& 16 & $6.76 \%$ & $\pm 5.03\%$ & $6.33\%$ & $\pm 4.43\%$ \\
\hline
$\mathbb{P}^{(\text{3D})}_{\text{nonunif}}$  & 4 & $33.4 \%$ & $\pm 35.4\%$ & $25.1\%$ & $\pm 21.6\%$ \\
& 8 & $13.9\%$ & $\pm 12.0\%$ & $1.22\%$ & $\pm 9.36\%$ \\
& 16 & $6.43\%$ & $\pm 5.12\%$ & $6.03\%$ & $\pm 4.53\%$ \\
\end{tabular} \label{tab:results_iso}
\caption{\textbf{Frobenius Error Comparison}: A performance comparison between the accuracy of the shifted operator estimator $\tilde{\truSol}$ for the solution $\truSol$ versus the naive estimator $\sampSol$. The value function is sampled from the prior $\mathcal{N}(\vc{0}, \id)$ (hence, the error is the Frobenius operator error in expectation). The error is measured as a percentage with respect to the average residual norm of the true solutions from the prior. Since the error is calculated via Monte Carlo, we provide a $95\%$ confidence interval in the $\pm 2\sigma$ column.}
\end{table}
\egroup

As we see in \cref{tab:results} and \cref{tab:results_iso}, operator shifting can provide significant reductions in error for a variety of different Markov chain problems, measured in the reduction of residual norm error for both deterministic and random value functions (\cref{tab:results}). As predicted by the theory, the method also reduces the error in isotropic residual matrix norm, as seen in (\cref{tab:results_iso}), but these improvements seem are more marginal. This behavior is present for all different levels of sample count $N$ we tested. Therefore, while there are theoretical limitations in the non-symmetric case of operator shifting that make the theory less powerful than the symmetric positive definite case, operator shifting still functions quite well on the Markov chain problems we've tested it on.

Note that the confidence intervals for the 2D and 3D problems in the isotropic setting seen in \cref{tab:results_iso} are quite large. This is despite the fact that we use a very large number of samples (256,000 for 2D, and 25,600 for 3D) to estimate the isotropic error. However, the theory tells us that the naive isotropic error will always be greater than the shifted isotropic error, so this chart is provided more-so as a way to gauge the magnitude of the error reduction, rather than existence of an error reduction.

\section{Conclusion}

We conclude this paper by noting that we have accomplished two main goals. First, we have
investigated the extent to which the symmetric positive definite operator shifting theory of
\cite{etter2020operator} can be applied to the general non-symmetric matrix case. We have found that
under the assumptions of noise symmetry and right-hand-side isotropy, the optimal shift factor is always positive. This answers the question of whether or not operator shifting towards the origin always reduces error as it does in the positive definite symmetric case. Moreover, we have fully characterized the pathological situations in which this does not happen. We have also investigated the small noise regime, where we showed that it is possible to discard the noise symmetry assumption.

Second, we have shown empirically that operator shifting can \emph{still} reduce error for noisy Markov chain problems, even when the aforementioned theoretical assumptions are not satisfied. This holds true across a number of different Markov chains and for both deterministic and isotropic right-hand-side. 

One may continue this work by attempting to apply some form of operator shifting to real problems -- for example, in control theory or reinforcement learning, where the underlying Markov decision process may not be fully known and must be estimated from data. Another more theoretical possibility would be to investigate if the operator shifting framework can be applied to optimization to create optimization algorithms that are less vulnerable to noise in the objective function, as is common in many real world applications. Other potentially interesting avenues of work may include extending operator shifting to an infinite dimensional setting, or trying to learn an appropriate shift $\augMat$ from data.

\section{Source Code}

For reproducibility and reference purposes, we provide an accompanying implementation of our algorithms and numerical experiments at \\

\begin{center}
    \url{https://github.com/UniqueUpToPermutation/OperatorShifting}.
\end{center}

\section{Acknowledgements}

The work of L.Y. is partially supported by the National Science Foundation under award DMS1818449 and DMS-2011699.


\appendix

\bibliographystyle{siamplain}
\bibliography{references}

\begin{thebibliography}{10}

\bibitem{anderson2010introduction}
{\sc G.~W. Anderson, A.~Guionnet, and O.~Zeitouni}, {\em An introduction to
  random matrices}, vol.~118, Cambridge university press, 2010.

\bibitem{aspri2020data}
{\sc A.~Aspri, Y.~Korolev, and O.~Scherzer}, {\em Data driven regularization by
  projection}, Inverse Problems, 36 (2020), p.~125009.

\bibitem{bertsekas2019reinforcement}
{\sc D.~Bertsekas}, {\em Reinforcement learning and optimal control}, Athena
  Scientific, 2019.

\bibitem{bleyer2013double}
{\sc I.~R. Bleyer and R.~Ramlau}, {\em A double regularization approach for
  inverse problems with noisy data and inexact operator}, Inverse Problems, 29
  (2013), p.~025004.

\bibitem{buccini2018semiblind}
{\sc A.~Buccini, M.~Donatelli, and R.~Ramlau}, {\em A semiblind regularization
  algorithm for inverse problems with application to image deblurring}, SIAM
  Journal on Scientific Computing, 40 (2018), pp.~A452--A483.

\bibitem{candes2010matrix}
{\sc E.~J. Candes and Y.~Plan}, {\em Matrix completion with noise}, Proceedings
  of the IEEE, 98 (2010), pp.~925--936.

\bibitem{etter2020operator}
{\sc P.~A. Etter and L.~Ying}, {\em Operator augmentation for noisy elliptic
  systems}, arXiv preprint arXiv:2010.09656,  (2020).

\bibitem{golub1980analysis}
{\sc G.~H. Golub and C.~F. Van~Loan}, {\em An analysis of the total least
  squares problem}, SIAM journal on numerical analysis, 17 (1980),
  pp.~883--893.

\bibitem{james1992estimation}
{\sc W.~James and C.~Stein}, {\em Estimation with quadratic loss}, in
  Breakthroughs in statistics, Springer, 1992, pp.~443--460.

\bibitem{keshavan2009matrix}
{\sc R.~Keshavan, A.~Montanari, and S.~Oh}, {\em Matrix completion from noisy
  entries}, in Advances in neural information processing systems, 2009,
  pp.~952--960.

\bibitem{lunz2021learned}
{\sc S.~Lunz, A.~Hauptmann, T.~Tarvainen, C.-B. Schonlieb, and S.~Arridge},
  {\em On learned operator correction in inverse problems}, SIAM Journal on
  Imaging Sciences, 14 (2021), pp.~92--127.

\bibitem{marzouk2016introduction}
{\sc Y.~Marzouk, T.~Moselhy, M.~Parno, and A.~Spantini}, {\em An introduction
  to sampling via measure transport}, arXiv preprint arXiv:1602.05023,  (2016).

\bibitem{palmer2005representing}
{\sc T.~Palmer, G.~Shutts, R.~Hagedorn, F.~Doblas-Reyes, T.~Jung, and
  M.~Leutbecher}, {\em Representing model uncertainty in weather and climate
  prediction}, Annu. Rev. Earth Planet. Sci., 33 (2005), pp.~163--193.

\bibitem{puterman2014markov}
{\sc M.~L. Puterman}, {\em Markov decision processes: discrete stochastic
  dynamic programming}, John Wiley \& Sons, 2014.

\bibitem{soize2005comprehensive}
{\sc C.~Soize}, {\em A comprehensive overview of a non-parametric probabilistic
  approach of model uncertainties for predictive models in structural
  dynamics}, Journal of sound and vibration, 288 (2005), pp.~623--652.

\bibitem{stein1956inadmissibility}
{\sc C.~Stein et~al.}, {\em Inadmissibility of the usual estimator for the mean
  of a multivariate normal distribution}, in Proceedings of the Third Berkeley
  symposium on mathematical statistics and probability, vol.~1, 1956,
  pp.~197--206.

\bibitem{tao2012topics}
{\sc T.~Tao}, {\em Topics in random matrix theory}, vol.~132, American
  Mathematical Soc., 2012.

\bibitem{tikhonov1963solution}
{\sc A.~N. Tikhonov}, {\em On the solution of ill-posed problems and the method
  of regularization}, in Doklady Akademii Nauk, vol.~151, Russian Academy of
  Sciences, 1963, pp.~501--504.

\bibitem{xiu2005high}
{\sc D.~Xiu and J.~S. Hesthaven}, {\em High-order collocation methods for
  differential equations with random inputs}, SIAM Journal on Scientific
  Computing, 27 (2005), pp.~1118--1139.

\bibitem{xiu2002wiener}
{\sc D.~Xiu and G.~E. Karniadakis}, {\em The wiener--askey polynomial chaos for
  stochastic differential equations}, SIAM journal on scientific computing, 24
  (2002), pp.~619--644.

\end{thebibliography}
\end{document}